\documentclass[12pt,reqno]{article}
\usepackage{hyperref}
\usepackage{amsmath, amsthm, graphicx,amsfonts, amssymb,color}
\usepackage{bm}
\usepackage[inline]{enumitem}
\usepackage{verbatim}
\usepackage{authblk}
\usepackage{geometry}

\newtheorem{thm}{Theorem}[section]
\newtheorem{cor}[thm]{Corollary}
\newtheorem{lem}[thm]{Lemma}
\newtheorem{prop}[thm]{Proposition}
\theoremstyle{definition}

\theoremstyle{remark}
\newtheorem{rem}[thm]{Remark}
\theoremstyle{example}

\numberwithin{equation}{section}
\usepackage{mathrsfs}

\def\de{\end{equation}}

\def\edar{\end{eqnarray}}

\def\l{\left}\def\r{\right}
\def\[{\l[} \def\]{\r]}
\def\({\l(} \def\){\r)}

\def\bar{\overline}

\def\beqlb{\begin{eqnarray}}\def\eeqlb{\end{eqnarray}}
\def\beqnn{\begin{eqnarray*}}\def\eeqnn{\end{eqnarray*}}
\def\d{{\mbox{\rm d}}}

\title{\bf  {Quasi-stationary distributions for single death processes with killing}}

\author[ ]{Zhe-Kang Fang\thanks{Corresponding author. \newline \indent \indent E-mail addresses: zkfang@mail.bnu.edu.cn (Z.-K. Fang), maoyh@bnu.edu.cn (Y.-H. Mao).}}
\author[ ]{Yong-Hua Mao}
\affil[ ]{\it \small School of Mathematical Sciences, Beijing Normal University, Laboratory of Mathematics and Complex Systems, Ministry of Education, Beijing 100875, China}

\date{}

\begin{document}

%%------------------------------------------------------------
\maketitle

%%------------------------------------------------------------

\begin{abstract}
This paper studies the quasi-stationary distributions for a single death process (or downwardly skip-free process) with killing defined on the non-negative integers, corresponding to a non-conservative transition rate matrix. The set $\{1,2,3,\cdots\}$ constitutes an irreducible class and $0$ is an absorbing state. For the single death process with three kinds of killing term, we obtain the existence and uniqueness of the quasi-stationary distribution. Moreover, we derive the conditions for exponential convergence to the quasi-stationary distribution in the total variation norm. Our main approach is based on the Doob's $h$-transform, potential theory and probabilistic methods.  
\end{abstract}



%%%%%%%%%%%%%%%%%%%%%%%%%%%%%%%%%%%%%%%%%%%%%%%%%%%%%%%%%%%%%%%%%%%%

\section{Introduction and main results}

Let $X=(X_t)_{t\geq 0}$ be a single death process (or downwardly skip-free process) with killing on $E=\{0,1,2,\cdots\}$ and irreducible on $E_+=\{1,2,3\cdots\}$, corresponding to a totally stable and non-conservative transition rate matrix $Q=(q_{ij})_{i,j\geq 0}$ satisfying
$q_i=-q_{ii}\geq \sum_{n\neq i} q_{in}$, $q_{0j}=0$, $q_{i,i-1}>0$, $q_{i,i-k}=0$, for any $i\geq 1, j\geq 0, k\geq 2$,
and transition function 
\begin{equation*}
p_{ij}(t)=\mathbb{P}_i[X_t=j,t<T_\partial],
\end{equation*}
for any $i,j\geq 0$, where $T_\partial=\inf\{t\geq 0: X_t\notin E\}$ denote the time for $X$ exit $E$ and we denote the first hitting time by $T_i=\inf\{t\geq 0:X_t=i\}$ for any $i\geq 0$. Obviously, $0$ is an absorbing state. We denote the killing term by $c=(c_n)_{n\geq 1}=(q_n-\sum_{j\neq n}q_{nj})_{n\geq 1}$. Assume that $c$ is not always 0.
If we add a cemetery point $\partial$ to the state space of $X$ and let $q_{i\partial}=c_i$, it can be regarded as a process on $E\cup\{\partial\}$ and $0$ and $\partial$ are two absorbing states. 

A quasi-stationary distribution (QSD in short) of the single death process with killing $X$ in $E_+$ is a probability measure on $E_+$ such that for any $j\geq 1$ and any $t>0$,
\begin{equation*}
\mathbb{P}_\nu[X_t=j|t<T_0\wedge T_\partial]=\nu(j).
\end{equation*}

QSD is a classical topic in the research of the long term behavior of Markov processes. Three main areas of study are identified: (1) determining the existence and uniqueness of the QSD; (2) identifying the QSD's domain of attraction; and (3) calculating the rate of convergence to the QSD. 

Birth and death process is a special case of single death process and the issue on the QSDs for birth and death processes with killing has attracted much attention. The existence of QSDs for birth and death processes with killing has been studied in \cite{V12}. 
For birth and death process with bounded killing ($\sup_{n\geq 1} c_n<\infty$), the sufficient and necessary condition for the uniform exponential convergence to a unique QSD in the total variation norm has been obtained using probabilistic methods by Champagnat and Villemonais in \cite{CV16}. In \cite{V18}, Velleret has provided a sufficient condition for exponential convergence in the total variation norm. He has also pointed out that at least for some of the models, the speed of convergence can not be uniformly bounded over all initial conditions in \cite[Theorem 4.1]{V18}.

The study of QSDs for single death processes without killing started with \cite{K93}. In recent years, \cite{Z18} and \cite{Y22} studied this topic and made significant progress. The uniform exponential convergence to QSD has been studied in \cite{Z18} by Y.-H. Zhang. In \cite{Y22}, Kosuke employed the potential theory of single death processes to study the existence and uniqueness of QSDs. By combining their results, we obtain that 
(1) the existence of QSD is equivalent to the exponential decay;
(2) the uniqueness of QSD is equivalent to the uniform exponential decay, which is also equivalent to the uniform exponential convergence to QSD for the conditional distributions in the total variation norm. From \cite{VE91}, these results are consistent with the results for birth and death processes. We will revisit their results in Section 3. The branching process with linear killing represents a specific instance of the single death process with killing, and the study referenced in \cite{CNZ14} delves into the QSD of this process by employing generating functions. However, there is no clear answer to this issue for the general single death process with killing.

Extensive research on the QSDs for birth and death processes with killing has revealed great differences between the killing case and the non-killing case, and the former is much more complex. Firstly, the exponential decay of $X$ may not imply the existence of QSD in the killing case. Secondly, according to \cite[Theorem 4.1]{V18}, the uniqueness of QSD for single death processes with killing does not imply the uniform exponential convergence in the total variation norm. 

The main purpose of this paper is to study the QSDs for single death processes with killing in $E_+$. In \cite{V12}, the killing term of birth and death processes are divided into two cases: ``small killing" case (\cite[Theorem 1]{V12}) and ``large killing" case (\cite[Theorem 2]{V12}). In this paper, we study three kinds of killing term. We study the ``small killing" case in Theorems \ref{exi_uni} and \ref{exp_con}, the ``bounded killing" case in Theorem \ref{uni_exp_con} and the ``large killing" case in Theorem \ref{large_killing}.

In our study, we use the non-killing process $Y=(Y_t)_{t\geq 0}$ corresponding to $X$, which is the minimal single death process on $E$ corresponding to the conservative transition rate matrix $Q^{(Y)}=(q_{ij}^{(Y)})_{i,j\geq 0}$ defined as 
\begin{equation*}
q_{ij}^{(Y)}=
\begin{cases}
	q_{ij} & \qquad \text{ for }i\neq j,\\
	-q_i+c_i & \qquad \text{ for } i=j.
\end{cases}
\end{equation*}
$T_i^{(Y)}=\inf\{t\geq 0:Y_t=i\}$ represents the hitting time of $i\geq 0$ for the process $Y$. It should be mentioned here that the QSDs for $X$ and $Y$ can be quite different. For example, \cite{ZZ13} proved that the linear birth and death process with linear killing has a unique QSD while the linear birth and death process without killing has a continuum of QSDs. 

To state our main results, we need some notations. Define
\begin{equation*}
G_n^{(n)}=1,\quad n\geq 1, \qquad G_n^{(k)}=\sum_{l=n+1}^{k}\frac{q_n^{(l)}G_l^{(k)}}{q_{n,n-1}}, \quad k>n\geq 1,
\end{equation*}
where
\begin{equation*}
q_n^{(k)}=\sum_{j=k}^{\infty}q_{nj}.
\end{equation*}
And for any $i,j\geq 0$,
\begin{equation*}
W(i,j)=\sum_{k=i+1}^{j}\frac{G_k^{(j)}}{q_{j,j-1}},
\end{equation*}
where we use the convention that $\sum_{\emptyset}=0$.

Our first theorem provides sufficient and necessary conditions for the existence of QSD for $X$ in $E_+$ under the ``small killing" condition: $\sum_{u=1}^{\infty} c(u)W(0,u)<\infty$. It is important to note that this condition is equivalent to
\begin{equation*}
\lim_{x\rightarrow\infty} \mathbb{P}_x[T_0<T_\partial]>0,
\end{equation*}
as will be demonstrated in Proposition \ref{small killing}.

\begin{thm}\label{exi_uni}
Let $X$ be a single death process with killing on $E$ and irreducible on $E_+$. Assume that $Y$ is the non-killing process corrsponding to $X$ and for any $i\geq 1$, $\mathbb{P}_i[T_0^{(Y)}<\infty]=1$. Assume that $\sum_{u=1}^{\infty} c(u)W(0,u)<\infty$. The existence of QSD for $X$ in $E_+$ is equivalent to the exponential decay of $X$, that is, for any $x\geq 1$,
\begin{equation*}
	\lambda_0^{(X)}=\sup\{\lambda>0:\mathbb{E}_x[e^{\lambda (T_0\wedge T_\partial)}]<\infty\}>0.
\end{equation*}
\end{thm}

The next theorem studies the uniqueness of QSD for $X$ in $E_+$ under the ``small killing" condition and shows that it is equivalent to the exponential convergence to QSD in the total variation norm.

\begin{thm}\label{exp_con}
Let $X$ be a single death process with killing on $E$ and irreducible on $E_+$. Assume that $Y$ is the non-killing process corresponding to $X$ and for any $i\geq 1$, $\mathbb{P}_i[T_0^{(Y)}<\infty]=1$. Assume that $\sum_{u=1}^{\infty} c(u)W(0,u)<\infty$. Then there exists a unique QSD $\nu$ for $X$ in $E_+$ if and only if
\begin{equation}\label{unique}
	\sum_{u=1}^{\infty}W(0,u)<\infty.
\end{equation}
If \eqref{unique} holds, then for any initial distribution $\mu$ on $E_+$, there exist constants $C(\mu),\gamma>0$ such that for any $t>0$,
\begin{equation*}
	||\mathbb{P}_\mu[X_t\in \cdot| t<T_0\wedge T_\partial]-\nu||_{TV}\leq C(\mu)e^{-\gamma t},
\end{equation*}
where $||\cdot||$ is the total variation norm.
\end{thm}

It is important to note that \eqref{unique} is equivalent to the coming down from infinity for $Y$, which is also consistent with the uniform exponential decay of $Y$ as discussed in Section 3. We will introduce a property of $X$ called ``coming down from infinity before killing" and show that it is equivalent to  ``$\sum_{u=1}^{\infty}(1+c(u))W(0,u)<\infty$" in Section 4.1.

The main tools used in the proof of the existence and uniqueness of QSD for $X$ in $E_+$ in above two theorems are Doob's $h$-transform and the results of QSDs for single death processes without killing. We obtain the explicit representation of the harmonic function $h$ for $X$ in Section 2 by studying the potential theory of $X$. Recently, the non-killing case has been studied by Yamato in \cite[Section 2]{Y22}. We generalize the exit identities to a $\omega$-killed version, for which previous results are special cases. We derive the exit identities by using the Poisson's equation established in \cite[Proposition 2.7]{CZ14}. The main tool to prove the exponential convergence in Theorem \ref{exp_con} is \cite[Theorem 2.1-2.3]{V18}. 

The following theorem provides a sufficient and necessary condition for the uniform exponential convergence of conditional distribution to QSD of single death processes with bounded killing ($\sup_i c_i<\infty$) in the total variation norm. The idea of the proof benefits from \cite{CV16}.

\begin{thm}\label{uni_exp_con}
Let $X$ be a single death process with killing on $E$ and irreducible on $E_+$. Assume that $Y$ is the non-killing process corrsponding to $X$ and for any $i\geq 1$, $\mathbb{P}_i[T_0^{(Y)}<\infty]=1$. Assume that $\sup_{n\geq 2} c_n<\infty$. There exists a QSD $\nu$ on $E_+$ for $X$ and two constants $C,\gamma>0$ such that for any initial distribution $\mu$ on $E_+$ and any $t\geq 0$,
\begin{equation*}
	||\mathbb{P}_\mu[X_t\in \cdot | t<T_0\wedge T_\partial]-\nu||_{TV}\leq C e^{-\gamma t},
\end{equation*}
if and only if $\sum_{u=1}^{\infty}W(0,u)<\infty$.
\end{thm}

The following theorem is about the result for the ``large killing" case and we obtain the existence and uniqueness of QSD and the exponential convergence to QSD.

\begin{thm}\label{large_killing}
Assume that for any $n\geq 1$, $\mathbb{P}_n[T_0\wedge T_\partial<\infty]=1$ and 
\begin{equation*}
	\liminf_{n\rightarrow \infty} c_n>\inf_{n\geq 1} q_n.
\end{equation*}
Then there exists a unique QSD $\nu$ for $X$ and for any initial distribution $\mu$ on $E_+$ and any $t>0$, there exist constants $C(\mu)$ and $\gamma$ such that
\begin{equation*}
	||\mathbb{P}_\mu[X_t\in \cdot| t<T_0\wedge T_\partial]-\nu||_{TV}\leq C(\mu)e^{-\gamma t}.
\end{equation*}
\end{thm}

The remaining part of this paper proceeds as follows. In Section 2, we study the potential theory of single death processes with killing. In Section 3, we revisit the result in the QSDs of single death processes without killing. In Section 4, we prove the main results.

\section{Potential theory for single death processes with killing}

In this section, the potential theory applicable to the single death process with killing is analyzed in order to derive an explicit representation of the harmonic function relevant to it. The exit problem for Markov processes has attracted significant attention in a variety of applied fields, including but not limited to mathematical finance, queuing theory, biology and physics. We generalize the exit identity to a $\omega$-killed version, for which previous results are special cases, and derive the exit identities by using the Poisson's equation established in \cite[Proposition 2.7]{CZ14}. 

Recall that $X$ is a single death process with killing and $Y$ is the non-killing process corresponding to $X$. Assume that $Y$ is unique. We can give a construction of $X$ by using $Y$ and $c=(c_n)_{n\geq 1}$. Define
\begin{equation*}
k_t=\int_{0}^{t} c(Y_s) \d s.
\end{equation*}
Let $\mathcal{E}$ be an exponential r.v. of parameter 1 independent of $Y$. Define
\begin{equation*}
T_\partial=\inf\{t\geq 0: k_t\geq \mathcal{E}\}.
\end{equation*}
The process $X$ is defined as
\begin{equation*}
X_t=
\begin{cases}
	Y_t &\text{ if } t< T_\partial\\
	\partial & \text{ otherwise }.
\end{cases}
\end{equation*}
Then the transition function for $X$ is 
\begin{equation*}
\mathbb{P}_i[X_t=j]=\mathbb{P}_i[Y_t=j,t<T_\partial]=\mathbb{E}_i\[\mathbf{1}_{[Y_t=j]}\exp\(-\int_{0}^{t} c(Y_s) \d s\)\], \qquad i,j\geq 0.
\end{equation*}
According to \cite[Section 5]{FP99}, it is the single death process with killing with transition rate matrix $Q$. It should be mentioned here that if $T_0^{(Y)}<T_\partial$, then $X$ is absorbed at $0$, and if $T_0^{(Y)}\geq T_\partial$, then $X$ is absorbed at $\partial$. According to this construction, we obtain that if the process $Y$ satisfies $\mathbb{P}_i[T_0^{(Y)}<\infty]=1$ for any $i\geq 1$, then the process $X$ satisfies $\mathbb{P}_i[T_0\wedge T_\partial <\infty]=1$ for any $i\geq 1$.

The following definitions generalize the notations used in the first section, which are intended to explicitly represent the results in potential theory.
Let $\omega=(\omega_i)_{i\geq 1}$ be a non-negative function on $E_+$. Define
\begin{equation}\label{G_omega}
G_n^{(n)}(\omega)=1,\quad n\geq 1, \qquad G_n^{(k)}(\omega)=\sum_{l=n+1}^{k}\frac{q_n^{(l)}(\omega)G_l^{(k)}(\omega)}{q_{n,n-1}}, \quad k>n\geq 1,
\end{equation}
where
\begin{equation*}
q_n^{(k)}(\omega)=\sum_{j=k}^{\infty}q_{nj}+\omega_n.
\end{equation*}
And for any $i,j\geq 0$,
\begin{equation}\label{W^omega}
W^{(\omega)}(i,j)=\sum_{k=i+1}^{j}\frac{G_k^{(j)}(\omega)}{q_{j,j-1}},
\end{equation}
where we use the convention that $\sum_{\emptyset}=0$.
And we define
\begin{equation*}
Z^{(\omega)}(i,j)=
\begin{cases}
	1+\sum_{i<u<j}\omega_u W^{(\omega)}(i,u) & \text{ if } i<j\\
	0 & \text{ if }i\geq j.
\end{cases}
\end{equation*}
According to the notations given in Section 1,  $G_n^{(k)}=G_n^{(k)}(0)$, $q_n^{(k)}=q_n^{(k)}(0)$ and $W(i,j)=W^{(0)}(i,j)$.

Our first lemma in this section is as follows. This lemma is obtained from \cite[Proposition 2.7]{CZ14}, which solves the Poisson equation for single death processes. We call a matrix $F$ is a upward triangle matrix if $F(i,j)=0$ for any $i\geq j$.

\begin{lem}\label{Poisson}
Let $\Omega=diag(0,\omega)$ denote the matrix which has $0$ as its first diagonal entry followed by the values of $\omega$ for the subsequent diagonal entries. Then $F=(W^{(\omega)}(i,j))_{i,j\geq 0}$ is the unique upward triangle matrix solution for 
\begin{equation}\label{W_solution}
	(Q^{(Y)}-\Omega)F(i,j)=\delta_{ij}, \qquad i\geq 1, j\geq 0.
\end{equation}
\end{lem}

\begin{proof}
Assume $F$ is the upward triangle matrix solution for \eqref{W_solution}. Fix $j\geq 1$ and let $g(i)=F(i,j)$, $f(i)=\delta_{ij}$, so $g(j+k)=0$ for any $k\geq 0$. From \eqref{W_solution}, we have $(Q^{(Y)}-\Omega)g(i)=f(i)$ for any $1\leq i\leq j$. According to \cite[Proposition 2.7]{CZ14}, for any $0\leq i<j$,
\begin{equation*}
	g(i)=\sum_{k=i+1}^{j}\sum_{l=k+1}^{j} \frac{G_k^{(j)}(\omega)f(l)}{q_{j,j-1}}=\sum_{k=i+1}^{j}\frac{G_k^{(j)}(\omega)}{q_{j,j-1}}=W^{(\omega)}(i,j).
\end{equation*}
So we complete the proof of this lemma.
\end{proof}

Recall that for any $a\geq 0$, $T_a=\inf\{t\geq 0:X_t=a\}$ denotes the first hitting time of $a$ for $X$,  $T_{a+}=\inf\{t\geq 0:X_t\geq a\}$ denotes the first passage time above $a$ for $X$ and $T_\partial=\inf\{t\geq 0:X_t\notin E\}$ denotes the first exit time from $E$ for $X$.

The following theorem solves the two-sided exit problem of $X$. We prove this theorem by using the result on Poisson's equation and strong Markov property.

\begin{thm}\label{two side exit}
Let $X$ be a single death process with killing on $E$ corresponding to $Q=(q_{ij})_{i,j\geq 0}$ with killing rate $c_i=-\sum_{j\geq i-1}q_{ij}$ for $i\geq 1$, and there exists $i_0\geq 1$ such that $c_{i_0}>0$. Let $\omega=(\omega_i)_{i\geq 1}$ be a non-negative function on $E_+$.

(1) For any $0\leq a<i<N$,
\begin{equation}\label{time}
	\mathbb{E}_i\[ \exp\(-\int_{0}^{T_a} \omega(X_t) \d t \), T_a<T_{N+}\wedge T_\partial\]=\frac{W^{(\omega+c)}(i,N)}{W^{(\omega+c)}(a,N)}.
\end{equation}

(2)
For any $0\leq a<i,j<N$,
\begin{equation}\label{measure}
	\mathbb{E}_i \[ \int_{0}^{T_a\wedge T_{N_+}\wedge T_\partial} \exp \! \( \! -\int_{0}^{t}  \omega(X_s) \d s \! \) \! \mathbf{1}_{[X_t=j]} \d t \]\!=\!\frac{ W^{(\omega+c)}(a,j) W^{(\omega+c)}(i,N)}{W^{(\omega+c)}(a,N)}-W^{(\omega+c)}(i,j).
\end{equation}

(3)
Then for any $0\leq a<i<N$,
\begin{equation}\label{time1}
	\begin{aligned}
		&\mathbb{E}_i\[ \exp\(-\int_{0}^{T_{N+}\wedge T_\partial} \omega(X_t) \d t\), T_{N+}\wedge T_\partial<T_a \]\\
		&=1+\sum_{i<j<N} \omega_j W^{(\omega+c)}(i,j)-\frac{W^{(\omega+c)}(i,N)}{W^{(\omega+c)}(a,N)}\( 1+\sum_{a<j<N}\omega_j W^{(\omega+c)}(a,j) \).
	\end{aligned}
\end{equation}
\end{thm}

\begin{proof}
(1) Firstly, we prove that for any $0<j<\infty$ and any $1\leq k\leq j$,
\begin{equation}\label{W_omega}
	W^{(\omega+c)}(j-k,j)=\frac{1}{q_{j,j-1}\mathbb{E}_{j-1}\[ \exp\(-\int_{0}^{T_{j-k}}\omega(X_t)\d t\), T_{j-k}<T_{j+}\wedge T_\partial \]}.
\end{equation}
We denote the right hand of the above equation by $M(j-k,j)$. By using strong Markov property and the skip-free property of $Y$, it follows that for any $1\leq l\leq k$,
\begin{align*}
	&\mathbb{E}_{j-1}\[ \exp\(-\int_{0}^{T_{j-k}}\omega(X_t)\d t\), T_{j-k}<T_{j+}\wedge T_\partial \]\\
	&=\mathbb{E}_{j-1}\[ \exp\(-\int_{0}^{T_{j-l}}\omega(X_t)\d t\), T_{j-l}<T_{j+}\wedge T_\partial \] \\
	&\qquad \times \mathbb{E}_{j-l}\[ \exp\(-\int_{0}^{T_{j-k}}\omega(X_t)\d t\), T_{j-k}<T_{j+}\wedge T_\partial \].
\end{align*} 
So
\begin{equation}\label{time_2}
	\mathbb{E}_{j-l}\[ \exp\(-\int_{0}^{T_{j-k}}\omega(X_t)\d t\), T_{j-k}<T_{j+}\wedge T_\partial \]=\frac{M(j-l,j)}{M(j-k,j)}.
\end{equation}
Using strong Markov property and skip-free property again, for $2\leq k\leq j-1$, we have
\begin{align*}
	&\mathbb{E}_{j-k+1}\[ \exp\(-\int_{0}^{T_{j-k}}\omega(X_t)\d t\), T_{j-k}<T_{j+}\wedge T_\partial \]\\
	&=\sum_{j-k+2\leq u\leq j-1} \frac{q_{j-k+1,u}}{q_{j-k+1}+\omega_{j-k+1}} \mathbb{E}_{u}\[ \exp\(-\int_{0}^{T_{j-k}}\omega(X_t)\d t\), T_{j-k}<T_{j+}\wedge T_\partial \] \\
	&\qquad +\frac{q_{j-k+1,j-k}}{q_{j-k+1}+\omega_{j-k+1}}.
\end{align*}
Combining with \eqref{time_2}, we get
\begin{equation}\label{W_omega3}
	(q_{j-k+1}+\omega_{j-k+1})M(j-k+1,j)=q_{j-k+1,j-k}M(j-k,j)+\sum_{j-k+2\leq u\leq j-1}q_{j-k+1,u}M(u,j). 
\end{equation}
By noting that $M(j-1,j)=\frac{1}{q_{j,j-1}}$.
Let $\Omega=\text{diag}(0,\omega)$ and $C=\text{diag}(0,c)$, and define $M(i,j)=0$ for any $i\leq j$. Obviously, $Q=Q^{(Y)}-C$. Then we have for $i\geq 1$ and $j\geq 0$,
\begin{equation*}
	(Q^{(Y)}-(\Omega+C))M(i,j)=\delta_{ij}.
\end{equation*}
According to Lemma \ref{Poisson}, we obtain the unique solution of the above equation, which shows that $M(j-k,j)=W^{(\omega+c)}(j-k,j)$. By using \eqref{time_2} again, we have \eqref{time}.

(2) Secondly, we prove \eqref{measure}. By using strong Markov property, we have that for any $0\leq a<i,j<N$,
\begin{align}\label{markov1}
	&\mathbb{E}_i\[ \int_{0}^{T_a\wedge T_{N+}\wedge T_\partial} \exp\(- \int_{0}^{t} \omega(X_s) \d s \)\mathbf{1}_{[X_t=j]}  \d t\] \nonumber \\
	&=\mathbb{E}_i \[ \exp\( -\int_{0}^{T_j} \omega(X_t)\d t  \), T_j<T_a\wedge T_{N+}\wedge T_\partial \] \nonumber \\
	&\qquad \times \mathbb{E}_j\[ \int_{0}^{T_a\wedge T_{N+}\wedge T_\partial} \exp\(- \int_{0}^{t} \omega(X_s) \d s \)\mathbf{1}_{[X_t=j]}  \d t\].
\end{align}

For any $0\leq a<j<i<N$,
\begin{align}\label{markov2}
	&\mathbb{E}_i \[ \exp\( -\int_{0}^{T_j} \omega(X_t)\d t  \), T_j<T_a\wedge T_{N+}\wedge T_\partial \] \nonumber\\
	&=\mathbb{E}_i \[ \exp\( -\int_{0}^{T_j} \omega(X_t)\d t  \), T_j<T_{N+}\wedge T_\partial \] \nonumber\\
	&=\frac{W^{(\omega+c)}(i,N)}{W^{(\omega+c)}(j,N)}.
\end{align}
And for any $a<i<j<N$,
\begin{align}\label{markov3}
	&\mathbb{E}_i \[ \exp\( -\int_{0}^{T_j} \omega(X_t)\d t  \), T_j<T_a\wedge T_{N+}\wedge T_\partial \] \nonumber \\
	&=\frac{\mathbb{E}_i\[ \exp\( -\int_{0}^{T_a} \omega(X_t) \d t \), T_j<T_a<T_{N+}\wedge T_\partial \]}{\mathbb{E}_j \[ \exp\( -\int_{0}^{T_a} \omega(X_t)\d t  \), T_a<T_{N+}\wedge T_\partial \]} \nonumber \\
	&=\frac{
		\mathbb{E}_i\[ \exp\( -\int_{0}^{T_a} \omega(X_t) \d t \), T_a<T_{N+}\wedge T_\partial \]-\mathbb{E}_i\[ \exp\( -\int_{0}^{T_a} \omega(X_t) \d t \), T_a<T_{j+}\wedge T_\partial  \]}{\mathbb{E}_j \[ \exp\( -\int_{0}^{T_a} \omega(X_t)\d t  \), T_a<T_{N+}\wedge T_\partial \]}\nonumber\\
	&=\frac{W^{(\omega+c)}(i,N)}{W^{(\omega+c)}(j,N)}-\frac{W^{(\omega+c)}(i,j)W^{(\omega+c)}(a,N)}{W^{(\omega+c)}(a,j)W^{(\omega+c)}(j,N)},
\end{align}
where the last equality comes from \eqref{time}.

So we just need to calculate 
\begin{equation*}
	\mathbb{E}_{j-l}\[ \int_{0}^{T_{j-k}\wedge T_{j+}\wedge T_\partial} \exp\(- \int_{0}^{t} \omega(X_s) \d s \)\mathbf{1}_{[X_t=j-l]}  \d t\]
\end{equation*}
for any $0<l<k\leq j$.

By using Markov property and \eqref{markov1},\eqref{markov2} and \eqref{markov3}, we obtain
\begin{align*}
	&\mathbb{E}_{j-l}\[ \int_{0}^{T_{j-k}\wedge T_{j+}\wedge T_\partial} \exp\(- \int_{0}^{t} \omega(X_s) \d s \)\mathbf{1}_{[X_t=j-l]}  \d t\]\\
	&=\frac{1}{q_{j-l}+\omega_{j-l}}+\mathbb{E}_{j-l}\[ \int_{0}^{T_{j-k}\wedge T_{j+}\wedge T_\partial} \exp\(- \int_{0}^{t} \omega(X_s) \d s \)\mathbf{1}_{[X_t=j-l]}  \d t\] \\
	&\qquad \times \left( \sum_{j-l+1 \leq u\leq j-1} \frac{q_{j-l,u}W^{(\omega+c)}(u,j)}{(q_{j-l}+\omega_{j-l})W^{(\omega+c)}(j-l,j)}\right.\\
	&\qquad \quad \left.+\frac{q_{j-l,j-l-1}}{q_{j-l}+\omega_{j-l}} \! \( \! \frac{W^{(\omega+c)}(j-l-1,j)}{W^{(\omega+c)}(j-l,j)}-\frac{W^{(\omega+c)}(j-l-1,j-l)W^{(\omega+c)}(j-k,j)}{W^{(\omega+c)}(j-k,j-l)W^{(\omega+c)}(j-l,j)} \! \) \! \right).
\end{align*}
By using \eqref{W_omega3}, we get that
\begin{equation*}
	\mathbb{E}_{j-l}\[ \int_{0}^{T_{j-k}\wedge T_{j+}\wedge T_\partial} \! \exp\(\!- \int_{0}^{t} \omega(X_s) \d s\! \) \!\mathbf{1}_{[X_t=j-l]}  \d t\]\!=\!\frac{W^{(\omega+c)}(j-k,j-l)W^{(\omega+c)}(j-l,j)}{W^{(\omega+c)}(j-k,j)}.
\end{equation*}
And \eqref{measure} is proved by using \eqref{markov1},\eqref{markov2} and \eqref{markov3}.

(3) Finally, we prove \eqref{time1}. On the one hand, according to \eqref{measure}, it follows that
\begin{align*}
	&\mathbb{E}_i\[ \int_{0}^{T_a\wedge T_{N+}\wedge T_\partial} \exp\(- \int_{0}^{t} \omega(X_s) \d s \)\omega(X_t)  \d t\]\\
	&=\frac{W^{(\omega+c)}(i,N)}{W^{(\omega+c)}(a,N)}\sum_{a<j<N} \omega_j W^{(\omega+c)}(a,j)
	-\sum_{i<j<N} \omega_j W^{(\omega+c)}(i,j).
\end{align*}
On the other hand, by using \eqref{time}
\begin{align*}
	&\mathbb{E}_i\[ \int_{0}^{T_a\wedge T_{N+}\wedge T_\partial} \exp\(- \int_{0}^{t} \omega(X_s) \d s \)\omega(X_t)  \d t\]\\
	&=1-\mathbb{E}_i\[\exp\(-\int_{0}^{T_a\wedge T_{N+}\wedge T_\partial} \omega(X_t) \d t \)\]\\
	&=1-\frac{W^{(\omega+c)}(i,N)}{W^{(\omega+c)}(a,N)}-\mathbb{E}_i\[\exp\(-\int_{0}^{T_{N+}\wedge T_\partial} \omega(X_t) \d t \),T_{N+}\wedge T_\partial<T_a\].
\end{align*}
We complete the proof by combining the above two equations.
\end{proof}

Recall that $Y=(Y_t)_{t\geq 0}$, the non-killing process corresponding to $X$, is a single death process without killing on $E$ with transition rate matrix $Q^{(Y)}=(q_{ij}^{(Y)})_{i,j\geq 0}$ satisfying $q_{ij}^{(Y)}=q_{ij}$ for any $i\neq j$. Recall that for any $a\geq 0$, $T_a^{(Y)}=\inf\{t\geq 0:Y_t=a\}$ denotes the first hitting time of a for $Y$ and $T_{a+}^{(Y)}=\inf\{t\geq 0:Y_t\geq a\}$ denotes the first passage time above $a$ for $Y$. The following theorem gives the solution to the two-side exit problem for $Y$. Employing the analogous technique as in the proof of Theorem \ref{two side exit}, we deduce the following corollary. Given the similarity in approach, we choose to omit the proof for brevity.

\begin{cor}
Let $Y=(Y_t)_{t\geq 0}$ be a single death process on $E$ corresponding to a conservative transition rate matrix $Q^{(Y)}=(q_{ij}^{(Y)})_{i,j\geq 0}$ satisfying $q_{ij}^{(Y)}=q_{ij}$ for any $i\neq j$, and $q_i^{(Y)}=\sum_{j\neq i}q_{ij}$. Let $\omega=(\omega_i)_{i\geq 1}$ be a non-negative function.

(1) 
For any $0\leq a<i<N$,
\begin{equation}\label{time2}
	\mathbb{E}_i\[ \exp\(-\int_{0}^{T_a^{(Y)}} \omega(Y_t) \d t \), T_a^{(Y)}<T_{N+}^{(Y)} \]=\frac{W^{(\omega)}(i,N)}{W^{(\omega)}(a,N)}.
\end{equation}

(2)
For any $0\leq a<i,j<N$,
\begin{equation}\label{measure2}
	\mathbb{E}_i \[ \int_{0}^{T_a^{(Y)}\wedge T_{N+}^{(Y)}} \exp\( -\int_{0}^{t}  \omega(Y_s) \d s \) \mathbf{1}_{[Y_t=j]} \d t \]=\frac{ W^{(\omega)}(a,j) W^{(\omega)}(i,N)}{W^{(\omega)}(a,N)}-W^{(\omega)}(i,j).
\end{equation}

(3)
For any $0\leq a<i<N$,
\begin{equation}\label{time4}
	\mathbb{E}_i\[ \exp\(-\int_{0}^{T_{N+}^{(Y)}} \omega(Y_t) \d t\), T_{N+}^{(Y)}<T_a^{(Y)} \]=Z^{(\omega)}(i,N)-\frac{W^{(\omega)}(i,N)Z^{(\omega)}(a,N)}{W^{(\omega)}(a,N)}.
\end{equation}
\end{cor}

The following theorem provides the exit identities for the one-sided exit problems associated with $Y$.

\begin{thm}\label{downward integral}
If the single death process $Y$ satisfies that for any $i\geq 1$, $\mathbb{P}_i[T_0^{(Y)}<\infty]=1$, then for any $i>a$,
\begin{equation}\label{limit}
	\mathbb{E}_i\[ \exp\(-\int_{0}^{T_a^{(Y)}} \omega(Y_t) \d t \)\] =\lim_{N\rightarrow\infty} \frac{W^{(\omega)}(i,N)}{W^{(\omega)}(a,N)}=\lim_{N\rightarrow\infty}\frac{Z^{(\omega)}(i,N)}{Z^{(\omega)}(a,N)}.
\end{equation}
If
\begin{equation*}
	C=\sum_{u=1}^{\infty} \omega_u W(0,u)<\infty,
\end{equation*}
then for any $i\geq 0$
\begin{equation}\label{Z_omega}
	Z^{(\omega)}(i,\infty):=1+\sum_{u=i+1}^{\infty}\omega_u W^{(\omega)}(i,u)<\infty,
\end{equation}
and
\begin{equation*}
	\mathbb{E}_i\[ \exp\(-\int_{0}^{T_a^{(Y)}} \omega(Y_t) \d t \)\] =\frac{Z^{(\omega)}(i,\infty)}{Z^{(\omega)}(a,\infty)}.
\end{equation*}
\end{thm}

\begin{proof}
The first equality of \eqref{limit} follows directly from \eqref{time2}. According to \eqref{time4}, we have
\begin{equation*}
	\lim_{N\rightarrow\infty} Z^{(\omega)}(i,N)-\frac{W^{(\omega)}(i,N)}{W^{(\omega)}(a,N)}Z^{(\omega)}(a,N)\leq \lim_{N\rightarrow\infty} \mathbb{P}_i[T_{N+}^{(Y)}<T_a^{(Y)}]=0,
\end{equation*}
which implies the second equality of \eqref{limit}.

In order to prove \eqref{Z_omega}, we first claim that
\begin{equation}\label{W W_omega_n}
	W^{(\omega)}(i,j)=\sum_{n=0}^{\infty} W^{(\omega,n)}(i,j),
\end{equation}
where
\begin{equation*}
	W^{(\omega,0)}(i,j)=W(i,j),
\end{equation*}
and for any $n \geq 0$,
\begin{equation*}
	W^{(\omega,n+1)}(i,j)=\sum_{i<u<j} W^{(\omega,n)}(i,u) \omega(u) W(u,j).
\end{equation*}
Indeed, by induction, it follows from \eqref{G_omega} that for any $l\geq k\geq 1$,
\begin{equation*}
	G_k^{(l)}(\omega)=G_k^{(l)}+\sum_{u=k}^{l-1}\sum_{j=u+1}^{l} \frac{G_k^{(u)}(\omega)\omega(u) G_j^{(l)}}{q_{u,u-1}},
\end{equation*} 
which indicates from \eqref{W^omega} that for any $i,l\geq 0$,
\begin{equation*}
	W^{(\omega)}(i,i+l+1)=W(i,i+l+1)+\sum_{u=i+1}^{i+l}W^{(\omega)}(i,u) \omega(u) W(u,i+l+1).
\end{equation*}
Now we prove \eqref{W W_omega_n} by induction. According to \eqref{W^omega}, for any $i\geq 0$ and $j=i+1$, \eqref{W W_omega_n} holds. Assume that for any $j\leq i+l$, \eqref{W W_omega_n} holds, and we prove for $j=i+l+1$ as follows
\begin{align*}
	\sum_{n=0}^{l} W^{(\omega,n)}(i,i+l+1)&=W(i,i+l+1)+\sum_{n=0}^{l-1} \sum_{u=i+1}^{i+l} W^{(\omega,n)}(i,u) \omega(u) W(u,i+l+1)\\
	&=W(i,i+l+1) +\sum_{n=0}^{l-1} \sum_{u=i+n+1}^{i+l} W^{(\omega,n)}(i,u)\omega(u)W(u,i+l+1)\\
	&=W(i,i+l+1) +\sum_{u=i+1}^{i+l} \sum_{n=0}^{u-i-1} W^{(\omega,n)}(i,u)\omega(u)W(u,i+l+1)\\
	&=W(i,i+l+1) +\sum_{u=i+1}^{i+l} W^{(\omega)}(i,u) \omega(u)W(u,i+l+1).
\end{align*}
So we prove this claim.

Then by using the method similar to \cite[Proposition 3.3]{Y22}, it follows from \eqref{W W_omega_n} that for any $n\geq 0$,
\begin{equation*}
	W^{(\omega,n)}(i,j)\leq \frac{W(i,j)}{n!}\(\sum_{i<u<j} \omega(u)W(i,u)\)^{n}.
\end{equation*} 
Obviously, for any $0\leq i<j$, $W(i,j)\leq W(0,j)$.
If $C=\sum_{u=1}^{\infty}\omega_u W(0,u)<\infty$, then
\begin{equation}\label{W_c W}
	W^{(\omega)}(i,j)\leq W(i,j)e^C,
\end{equation}
which implies that for any $0\leq i<j<\infty$,
\begin{equation*}
	Z^{(\omega)}(i,j)\leq 1+Ce^C<\infty.
\end{equation*}
So $Z^{(\omega)}(i,\infty)<\infty$. And we complete the proof by using \eqref{limit}.
\end{proof}

Based on the above two theorems, we come to the following corollary. In Section 4, we will show that \eqref{h(i)} is a harmonic function for $X$, which plays an important role in the proof of our main results.

\begin{cor}
Assume the process $Y$ satisfy for any $i\geq 1$, $\mathbb{P}_i[T_0^{(Y)}<\infty]=1$. Assume $\sum_{u=1}^{\infty} c_u W(0,u)<\infty$, then
\begin{equation}\label{h(i)}
	\mathbb{P}_x[T_0<T_\partial]=\frac{Z^{(c)}(x,\infty)}{Z^{(c)}(0,\infty)},
\end{equation}
and
\begin{equation*}
	G^0(x,y)=\mathbb{E}_x\[\int_{0}^{T_0\wedge T_\partial}\mathbf{1}_{[X_t=y]}\d t\]=\frac{W^{(c)}(0,y)Z^{(c)}(x,\infty)}{Z^{(c)}(0,\infty)}-W^{(c)}(x,y).
\end{equation*}
If $\sum_{u=1}^{\infty}W(0,u)<\infty$ also holds, then for any $q> 0$, $0<i,j<\infty$ and $i,j\neq \partial$,
\begin{equation*}
	\mathbb{E}_i \[\int_{0}^{T_0\wedge T_\partial} e^{-qt} \mathbf{1}_{[X_t=j]} \d t \]=\frac{W^{(q+c)}(0,j)Z^{(q+c)}(i,\infty)}{Z^{(q+c)}(0,\infty)}-W^{(q+c)}(i,j),
\end{equation*}
where $q+c$ denotes the function $(q_i+c_i)_{i\geq 1}$.
\end{cor}

\begin{proof}
(1) According to Theorem \ref{two side exit} (1), we arrive that
\begin{equation*}
	\mathbb{P}_i[T_0<T_{N+}\wedge T_\partial]=\frac{W^{(c)}(i,N)}{W^{(c)}(0,N)}.
\end{equation*}
Since $Y$ is unique, we get that $X$ is unique, which implies that
for any $i\geq 1$,
\begin{equation*}
	\mathbb{P}_i[T_0<T_\partial]=\lim_{N\rightarrow\infty} \mathbb{P}_i[T_0<T_{N+}\wedge T_\partial]=\lim_{N\rightarrow\infty} \frac{W^{(c)}(i,N)}{W^{(c)}(0,N)}.
\end{equation*}
According to Theorem \ref{downward integral}, we have
\begin{equation*}
	\lim_{N\rightarrow\infty} \frac{W^{(c)}(i,N)}{W^{(c)}(0,N)}=\lim_{N\rightarrow\infty}\frac{Z^{(c)}(i,N)}{Z^{(c)}(0,N)}=\frac{Z^{(c)}(x,\infty)}{Z^{(c)}(0,\infty)},
\end{equation*}
which completes the proof of the first argument.

(2) The second and third arguments can be proved by using Theorem \eqref{two side exit}(2) and an argument similar to (1), so we omit their proofs. 
\end{proof}

\section{QSDs for single death processes without killing}

In this section, we revisit the results for the QSDs for single death processes without killing with absorbing state $0$. We prove that the uniform exponential decay is equivalent to the uniqueness of QSD.  

Recall that $Y$ is a totally stable single death process on $E$, where $0$ is an absorbing state and  $E\setminus\{0\}$ forms an irreducible class. We denote the decay parameter of $Y$ by 
\begin{equation*}
\lambda_0^{(Y)}=\sup \{\lambda\geq 0: \mathbb{E}_x[e^{\lambda T_0^{(Y)}}]<\infty\},
\end{equation*}
which is independent of the choice of $x>0$. According to \cite[Proposition 4.12]{CMS13}, for any $i,j\geq 1$,
\begin{equation}\label{lambda_Y}
\lambda_0^{(Y)}=-\lim_{t\rightarrow\infty} \frac{1}{t} \log p_{ij}^{(Y)}(t),
\end{equation}
where $p_{ij}^{(Y)}(t)=\mathbb{P}_i[Y_t=j,t<T_0^{(Y)}]$.

We combine the theorems on QSDs for single death processes in \cite{Z18} and \cite{Y22}, and obtain the following results, which indicate that there are three possible scenarios for the QSD of a single death process without killing. The following proposition shows that if $Y$ comes down from infinity, then there exists a unique QSD for $Y$ and the uniform exponential convergence to QSD in the total variation norm holds.

\begin{prop}\label{no-killing unique}
Assume that for any $n\geq 0$, $\mathbb{P}_n[T_0^{(Y)}<\infty]=1$. Then the following statements are equivalent.
\begin{enumerate}[label=(\arabic*)]
	\item $S=\sum_{u=1}^{\infty}W(0,u)=\sum_{n\geq 1} \sum_{k\geq n} \frac{G_n^{(k)}}{q_{k,k-1}} <\infty$, that is, $\sup_{n\geq 1} \mathbb{E}_n[T_0^{(Y)}]<\infty$.
	
	\item $Y$ comes down from infinity, that is, there exists $t_0>0$ such that
	\begin{equation*}
		\lim_{x\rightarrow\infty} \mathbb{P}_x[T_0^{(Y)}<t_0]>0.
	\end{equation*}
	
	\item There exists a unique QSD $\nu$ for $Y$.
	
	\item 	There exists a unique QSD $\nu$ for $Y$ such that for any probability measure $\mu$ on $E_+$,
	\begin{equation*}
		\Vert \mathbb{P}_\mu (Y_t \in \cdot | t<T_0^{(Y)}) -\nu \Vert_{TV} \leq 2(1-\gamma)^{[t]},\quad  t\geq 0,
	\end{equation*}
	where $\gamma$ is a positive constant independent of $\mu$, and the decay parameter of $\nu$ is $\lambda_0^{(Y)}$.
	
	\item $Y$ is uniform exponential decay, and the decay parameter is equal to the exponential decay parameter, that is,
	\begin{equation*}
		\lambda_0^{(Y)}=-\lim_{t\rightarrow+\infty} \frac{1}{t} \log \sup_i \mathbb{P}_i[T_0^{(Y)}>t]>0.
	\end{equation*}
\end{enumerate}
Furthermore if any of the statements holds, the QSD $\nu$ can be represented as follows:
\begin{equation*}
	\nu_i=\lambda_0^{(Y)}W^{(-\lambda_0^{(Y)})}(0,i), \quad i\geq 1.
\end{equation*}
\end{prop}

\begin{proof}
Obviously, $(4)\Rightarrow(3)$ holds. $(1)\Rightarrow(4)$ has been proved in \cite{Z18}. $(3)\Rightarrow(1)$ has been proved in \cite{Y22}. So it remains to prove $(1)\Leftrightarrow(2)$, $(4)\Rightarrow(5)$ and $(5)\Rightarrow(2)$.

Firstly, we prove $(1)\Rightarrow(2)$. We show that if $\sup_{n\geq 1} \mathbb{E}_n[T_0^{(Y)}]<\infty$, then for any $\epsilon>0$, there exists $t_{\epsilon}>0$ such that $\lim_{x\rightarrow\infty}\mathbb{P}_x[T_0^{(Y)}>t_{\epsilon}]<\epsilon$.  

By using strong Markov property and skip-free property of $Y$, we get that for any $t>0$ and $n\geq 1$, $\mathbb{P}_n[T_0^{(Y)}>t]\leq \mathbb{P}_{n+1}[T_0^{(Y)}>t]$. So $\lim_{n\rightarrow \infty} \mathbb{P}_n[T_0^{(Y)}>t]$ exists. By using Fubini theorem and monotone convergence theorem, 
\begin{equation*}
	\lim_{n\rightarrow \infty} \mathbb{E}_n[T_0^{(Y)}]=\lim_{n\rightarrow \infty} \int_{0}^{\infty} \mathbb{P}_n[T_0^{(Y)}>t] \d t=\int_{0}^{\infty} \lim_{n\rightarrow \infty}\mathbb{P}_n[T_0^{(Y)}>t] \d t<\infty.
\end{equation*}
Since $\lim_{n\rightarrow \infty} \mathbb{P}_n[T_0^{(Y)}>t]$ is a decreasing function respect to $t$, we have
\begin{equation*}
	\lim_{t\rightarrow \infty} \lim_{n\rightarrow \infty} \mathbb{P}_n[T_0^{(Y)}>t]=0,
\end{equation*}
which implies the existence of $t_\epsilon$. So $(1)\Rightarrow(2)$.

Secondly, we prove $(2)\Rightarrow(1)$. Given that $Y$ satisfies for any $i\geq 1$, $\mathbb{P}_i[T_0^{(Y)}<\infty]=1$,  
it follows from \cite[Proposition 2.5]{Z18} or \cite[Proposition A.1]{Y22} that $S=\sup_{n\geq 1}\mathbb{E}_n[T_0^{(Y)}]=\sum_{u=1}^{\infty} W(0,u)$.

Suppose that
\begin{equation*}
	\lim_{n\rightarrow \infty} \mathbb{P}_n[T_0^{(Y)}>t_0]=c\in [0,1).
\end{equation*}
By using Markov property, for any $k\geq 1$, we have
\begin{equation*}
	\mathbb{P}_n[T_0^{(Y)}>kt_0]\leq c^k.
\end{equation*}
So by using Fubini theorem,
\begin{equation*}
	\sup_{n\geq 1}\mathbb{E}_n[T_0^{(Y)}]\leq t_0 \sum_{k\geq 0} \sup_{n\geq 1} \mathbb{P}_n[T_0>k t_0] \leq \frac{t_0}{1-c}<\infty.
\end{equation*}

Thirdly, we prove $(4)\Rightarrow(5)$. We claim that for any $\lambda<\lambda_0^{(Y)}$,
\begin{equation}\label{lambda T_0 Y}
	\sup_{n\geq 1}\mathbb{E}_n[e^{\lambda T_0^{(Y)}}]<\infty.
\end{equation}
Indeed, from the statement (4), we have that for any $t>0$, 
$\mathbb{P}_\nu(T_0^{(Y)}>t)= e^{-\lambda_0^{(Y)} t}$,
and for any probability measure $\mu$ on $E_+$, and any bounded, measurable function $f$ on $E_+$,
\begin{equation*}
	\mathbb{E}_\mu \[f(Y_t)|t<T_0^{(Y)}\] = \frac{\mathbb{E}_\mu\[f(Y_t)\mathbf{1}_{[t<T_0^{(Y)}]}\]}{\mathbb{P}_\mu\[t<T_0^{(Y)}\]} \rightarrow \nu(f) \quad \text{as} \ t\rightarrow \infty.		
\end{equation*}
Let $f(x)=\mathbb{P}_x[T_0^{(Y)}>1]$, then $\nu(f)=e^{-\lambda_0^{(Y)}}$.
According to Markov property, we have
$\lim_{t\rightarrow \infty} \mathbb{P}_\mu[T_0^{(Y)}>t+1]/\mathbb{P}_\mu[T_0^{(Y)}>t] = e^{-\lambda_0^{(Y)}}$.
Then we take $\epsilon >0 $ such that $(1+\epsilon)e^{\lambda-\lambda_0^{(Y)}}<1$. 
So there exists $t_0 \in \mathbb{N} $ such that for any $t\geq t_0$, $\mathbb{P}_\mu[T_0^{(Y)}>t+1] \leq e^{-\lambda_0^{(Y)}}(1+\epsilon)\mathbb{P}_\mu[T_0^{(Y)}>t]$. By induction, we get that for any $n\geq 1$, $\mathbb{P}_\mu[T_0^{(Y)}>t+n] \leq e^{-\lambda_0^{(Y)} n}(1+\epsilon)^n\mathbb{P}_\mu[T_0^{(Y)}>t]$.
Then by using Fubini theorem,
\begin{align*}
	\mathbb{E}_\mu [e^{\lambda T_0^{(Y)}}]&=\int_{0}^{\infty} e^{\lambda t} \mathbb{P}_\mu[T_0^{(Y)}\in  \d t] =\lambda \int_{0}^{\infty} e^{\lambda t} \mathbb{P}_\mu[T_0^{(Y)}>t] \d t+1\\
	&\leq \lambda t_0 e^{\lambda t_0}+\lambda \sum_{k=t_0}^{\infty} e^{\lambda (k+1)} \mathbb{P}_\mu[T_0^{(Y)}>k]+1\\
	&\leq \lambda t_0 e^{\lambda t_0}+C_1\sum_{k=t_0}^{+\infty} e^{(\lambda-\lambda_0^{(Y)})k}(1+\epsilon)^k +1<\infty,
\end{align*}
where $C_1$ is a constant. It comes to a conclusion that for any probability measure $\mu$ on $E_+$, and any $\lambda \in (0,\lambda_0^{(Y)})$, $\mathbb{E}_\mu[e^{\lambda T_0^{(Y)}}] <\infty$. Then let $g(n)=\mathbb{E}_n [e^{\lambda T_0^{(Y)}}]$. We claim that $g$ is bounded. If not, there would exist sequences $\{x_n\}$ such that $g(x_n)\geq 2^n$, $n \geq 1$. However, if we take $\eta=\sum_{n=1}^{\infty} \frac{1}{2^n} \delta_{x_n}$, where $\delta_x$ is the dirac measure, then it is easy to verify that $\mathbb{E}_\eta[e^{\lambda T_0^{(Y)}}]=+\infty$, a contradiction. So we obtain \eqref{lambda T_0 Y}. 

By using Chebyshev inequality, we have for any $\lambda<\lambda_0^{(Y)}$,
$\sup_n \mathbb{P}_n[T_0^{(Y)}>t] \leq e^{-\lambda t} \sup_n \mathbb{E}_n[e^{\lambda T_0^{(Y)}}]$.
Then 
\begin{equation*}
	\liminf_{t\rightarrow +\infty }-\frac{1}{t} \log \sup_n \mathbb{P}_n[T_0^{(Y)}>t]\geq \lambda.
\end{equation*}
So 
\begin{equation*}
	\liminf_{t\rightarrow +\infty }-\frac{1}{t} \log \sup_n \mathbb{P}_n[T_0^{(Y)}>t]\geq \lambda_0^{(Y)}.
\end{equation*}
Since for any $i,j\geq 1$,
$ \sup_n \mathbb{P}_n[T_0^{(Y)}>t]\geq p_{ij}^{(Y)}(t)$,
it follows from \eqref{lambda_Y} that 
\begin{equation*}
	\limsup_{t\rightarrow +\infty }-\frac{1}{t} \log \sup_n \mathbb{P}_n[T_0^{(Y)}>t]\leq \lambda_0^{(Y)},
\end{equation*}
which completes the proof of $(4)\Rightarrow(5)$.

Finally, we prove $(5)\Rightarrow(2)$. It follows from (5) that there exists $t_0>0$, 
\begin{equation*}
	\sup_{n\geq 1} \mathbb{P}_n[T_0^{(Y)}>t_0]<1.
\end{equation*}
By using the strong Markov property and the skip-free property, for any $n\geq 1$,
\begin{equation*}
	\mathbb{P}_n[T_0^{(Y)}>t_0]\leq \mathbb{P}_{n+1} [T_0^{(Y)}>t_0],
\end{equation*}
which implies that 
\begin{equation*}
	\lim_{n\rightarrow \infty} \mathbb{P}_n[T_0^{(Y)}<t_0]>0.
\end{equation*}
\end{proof}

The next lemma is from \cite{Y22}, which implies that if $Y$ is exponential decay but not coming down from infinity, then there exists a continuum of QSDs for $Y$. 

\begin{lem}\label{QSD_Y}
Assume that for any $n\geq 0$, $\mathbb{P}_n[T_0^{(Y)}<\infty]=1$. There exists a continuum of QSDs $\{\nu^{\theta}\}_{\theta \in (0,\lambda_0^{(Y)}]}$ for $\{Y_t: t\geq 0\}$ if and only if $S=\infty$ and $\lambda_0^{(Y)}>0$, where
\begin{equation*}
	\nu^{\theta}_i=\theta W^{(-\theta)}(0,i), \quad i\geq 1.
\end{equation*}
\end{lem}

We note it here that if $Y$ is not exponential decay, from \cite[Proposition 2.4]{CMS13}, there exists no QSD for $Y$.

\section{Proof of the main results}

In this section, we are going to prove our main results.  Firstly, we consider the ``small killing" case and prove Theorems \ref{exi_uni} and \ref{exp_con}. Secondly, we study the ``bounded killing" case and prove Theorem \ref{uni_exp_con}. Finally, we study the ``large killing" case, which corresponds to Theorem \ref{large_killing}.

\subsection{``Small killing" case: proofs of Theorem \ref{exi_uni} and \ref{exp_con}}

This subsection presents the results for the ``small killing" case, corresponding to Theorems \ref{exi_uni} and \ref{exp_con}. We introduce a property called ``coming down from infinity before killing" for $X$ and give an explicit characterization to it. In the final part of this subsection, we offer a probabilistic interpretation of the ``small killing" condition. In this subsection, we assume that for the non-killing process $Y$ corresponding to $X$, $\mathbb{P}_i[T_0^{(Y)}<\infty]=1$ for any $i\geq 1$.

As stated in Section 1, the main tools used in the proofs of Theorems \ref{exi_uni} and \ref{exp_con} are Doob's $h$-transform, probabilistic methods and the results of QSDs for single death processes without killing. The following proposition provides the construction of Doob's $h$-transformed process $\bar{X}$ and characterizes the basic properties of it. 

\begin{prop}\label{h process}
Suppose the assumptions in Theorem \ref{exi_uni} hold. Let
\begin{equation*}
	h(i)=\mathbb{P}_i[T_0<T_\partial]=\frac{Z^{(c)}(i,\infty)}{Z^{(c)}(0,\infty)}.
\end{equation*}
Then it is a harmonic function for $X$, that is, for any $i\geq 0$,
\begin{equation*}
	P(t)h(i)=\sum_{j=0}^{\infty} p_{ij}(t)h(j)=h(i).
\end{equation*}
Define $\bar{Q}=(\bar{q}_{ij})_{i,j\geq 0}$ with
\begin{equation*}
	\bar{q}_{ij}=\frac{q_{ij}h(j)}{h(i)}.
\end{equation*}
Let $\bar{X}$ be the minimal process corresponding to $\bar{Q}$, then its transition function $\bar{P}(t)=(\bar{p}_{ij}(t))_{i,j\geq 0}$ satisfying
\begin{equation}\label{bar{P}(t)}
	\bar{p}_{ij}(t)=\frac{p_{ij}(t)h(j)}{h(i)},
\end{equation}
and $\bar{X}$ is conservative on $E$, and for any $i\geq 1$, $\mathbb{P}_i[T_0^{(\bar{X})}<\infty]=1$,
where $T_0^{(\bar{X})}=\inf\{t\geq 0: \bar{X}(t)=0\}$.
\end{prop}

\begin{proof}
(1) Firstly, we prove that $P(t)h(i)=h(i)$ for $h(i)=\mathbb{P}_i[T_0<T_\partial]$. Obviously, $h(0)=1$. Since $0$ is an absorbing state for $X$, we have that for any $i,j\geq 1$,
\begin{equation*}
	p_{ij}(t)=\mathbb{P}_i[X_{t\wedge T_0}=j, t\wedge T_0<T_\partial].
\end{equation*} 
By using strong Markov property, we obtain that
\begin{align*}
	P(t)h(i)&=\mathbb{E}_i[h(X_{t\wedge T_0})\mathbf{1}_{[t\wedge T_0<T_\partial]}]\\
	&=\mathbb{E}_i[h(0) \mathbf{1}_{[T_0<t, T_0<T_\partial]}]+\mathbb{E}_i[h(X_t)\mathbf{1}_{[t<T_0\wedge T_\partial]}]\\
	&=\mathbb{P}_i[T_0<t,T_0<T_\partial]+\mathbb{E}_i[\mathbb{E}_i[\mathbf{1}_{[T_0<T_\partial]}|X_t]\mathbf{1}_{[t<T_0\wedge T_\partial]}]\\
	&=\mathbb{P}_i[T_0<t,T_0<T_\partial]+\mathbb{P}_i[t<T_0<T_\partial]\\
	&=\mathbb{P}_i[T_0<T_\partial]=h(i).
\end{align*}
So $h$ is the harmonic function for $X$. The second equality is derived from \eqref{h(i)} directly.

(2) Secondly, we prove \eqref{bar{P}(t)} by using the second successive approximation scheme. For any $\lambda>0$, let 
\begin{equation*}
	p_{ij}(\lambda)=\int_{0}^{\infty} e^{-\lambda t} p_{ij}(t) \d t,
\end{equation*}
and 
\begin{equation*}
	\bar{p}_{ij}(\lambda)=\int_{0}^{\infty} e^{-\lambda t}\bar{p}_{ij}(t) \d t,
\end{equation*}
where $\bar{p}_{ij}(t)$ denote the transition function for the minimal process corresponding to $\bar{Q}$.
It remains to prove that
\begin{equation*}
	\bar{p}_{ij}(\lambda)=p_{ij}(\lambda)\frac{h(j)}{h(i)}.
\end{equation*}
Since the minimal process is the unique solution to the backward equation, we obtain that
\begin{equation*}
	p_{ij}(\lambda)=\sum_{k\neq i} \frac{q_{ik}}{\lambda+q_i} p_{kj}(\lambda)+\frac{\delta_{ij}}{\lambda+q_i},
\end{equation*}
and 
\begin{equation*}
	\bar{p}_{ij}(\lambda)=\sum_{k\neq i} \frac{q_{ik}}{\lambda+q_i}\frac{h(k)}{h(i)} \bar{p}_{kj}(\lambda)+\frac{\delta_{ij}}{\lambda+q_i}.
\end{equation*}
By the second successive approximation scheme, let
\begin{equation*}
	p_{ij}^{(1)}(\lambda)=\frac{\delta_{ij}}{\lambda+q_i}, \qquad p_{ij}^{(n+1)}(\lambda) =\sum_{k\neq i} \frac{q_{ik}}{\lambda+q_i} p^{(n)}_{kj}(\lambda),
\end{equation*}
and 
\begin{equation*}
	\bar{p}_{ij}^{(1)}(\lambda)=\frac{\delta_{ij}}{\lambda+q_i}, \qquad \bar{p}_{ij}^{(n+1)}(\lambda) =\sum_{k\neq i} \frac{q_{ik}}{\lambda+q_i} \frac{h(k)}{h(i)} \  \bar{p}^{(n)}_{kj}(\lambda).
\end{equation*}
By induction, it is easy to show that for any $n\geq 0$,
\begin{equation*}
	\bar{p}^{(n)}_{ij}(\lambda)=p^{(n)}_{ij}(\lambda) \frac{h(j)}{h(i)}.
\end{equation*}
Then it follows that 
\begin{equation*}
	\bar{p}_{ij}(\lambda)=\sum_{n=0}^{\infty}\  \bar{p}^{(n)}_{ij}(\lambda)=\sum_{n=0}^{\infty} p^{(n)}_{ij}(\lambda) \frac{h(j)}{h(i)}=p_{ij}(\lambda) \frac{h(j)}{h(i)},
\end{equation*}
which implies \eqref{bar{P}(t)}.

(3) Finally, we prove that $\bar{X}$ is unique and conservative on $E$ and it holds that for any $i\geq 1$,
$\mathbb{P}_i[T_0^{(\bar{X})}<\infty]=1$.

Firstly, it follows from $P(t)h=h$ that $\sum_{j=0}^{\infty}\bar{p}_{ij}(t)=\sum_{j=0}^{\infty} p_{ij}(t)\frac{h(j)}{h(i)}=1$, which implies that the minimal process is conservative. So $\bar{X}$ is unique.

Secondly, according to \cite[Theorem 1.5(1)]{MYZ22}, it remains to prove that
\begin{equation}\label{killing}
	\lim_{N\rightarrow\infty} \frac{\bar{G}_1^{(N)}}{\sum_{k=1}^{N} \bar{G}_k^{(N)}}=0,
\end{equation}
where $\bar{G}_k^{(N)}(q)$ is defined by replacing $q_{ij}$ with $\bar{q}_{ij}$ and taking $\omega \equiv q$ in \eqref{G_omega} and setting $\bar{G}_k^{(N)}=\bar{G}_k^{(N)}(0)$. 

We claim that for any $N\geq k\geq 1$ and any $q\in \mathbb{R}$,
\begin{equation}\label{bar(G)}
	\bar{G}_k^{(N)}(q)=\frac{q_{N,N-1}W^{(q+c)}(k-1,N)Z^{(c)}(N-1,\infty)}{Z^{(c)}(k-1,\infty)}-\frac{q_{N,N-1}W^{(q+c)}(k,N)Z^{(c)}(N-1,\infty)}{Z^{(c)}(k,\infty)}.
\end{equation}
We prove this claim by induction. For any $n\geq 1$, $\bar{G}_n^{(n)}(q)=1$, so \eqref{bar(G)} holds for $N=k$.
Assume for any $n\geq 1$ and any $0\leq i\leq k$, \eqref{bar(G)} holds, that is,
\begin{align*}
	\bar{G}_n^{(n+i)}(q)&=\frac{q_{n+i,n+i-1}W^{(q+c)}(n-1,n+i)Z^{(c)}(n+i-1,\infty)}{Z^{(c)}(n-1,\infty)}\\
	&\qquad -\frac{q_{n+i,n+i-1}W^{(q+c)}(n,n+i)Z^{(c)}(n+i-1,\infty)}{Z^{(c)}(n,\infty)},
\end{align*}
and we prove for $i=k+1$. Since $P(t)h=h$, we have $Qh=0$. It follows that for any $k\geq n+1$,
\begin{equation*}
	\bar{q}_n^{(k)}(q)=\sum_{j=k}^{\infty} \bar{q}_{nj}+q=\sum_{j=k}^{\infty} q_{nj} \frac{h(j)}{h(n)}+q=q_n+q- \frac{q_{n,n-1} h(n-1)}{h(n)}-\sum_{j=n+1}^{k-1} q_{nj}\frac{h(j)}{h(n)}.
\end{equation*}
According to the definition of $\bar{G}_n^{(n+k+1)}(q)$, we have
\begin{align*}
	\bar{G}_n^{(n+k+1)}(q)&=\sum_{l=n+1}^{n+k+1}\frac{ \bar{q}_n^{(l)}(q) \bar{G}_l^{(n+k+1)}(q) }{\bar{q}_{n,n-1}}\\
	&=\sum_{l=n+1}^{n+k+1} \frac{(q_n+q)h(n)}{q_{n,n-1}h(n-1)} \bar{G}_l^{(n+k+1)}(q)-\sum_{l=n+1}^{n+k+1} \bar{G}_l^{(n+k+1)}(q)\\
	&\qquad -\sum_{l=n+1}^{n+k+1} \sum_{j=n+1}^{l-1} \frac{q_{nj} h(j)}{q_{n,n-1} h(n-1)} \bar{G}_l^{(n+k+1)}(q).
\end{align*}
By assumption, it follows that
\begin{align*}
	\bar{G}_n^{(n+k+1)}(q)&=\frac{(q_n+q)q_{n+k+1,n+k}W^{(q+c)}(n,n+k+1)h(n+k)}{q_{n,n-1}h(n-1)}\\
	&\qquad -\frac{q_{n+k+1,n+k}W^{(q+c)}(n,n+k+1)h(n+k)}{h(n)}\\
	&\qquad -\sum_{l=n+1}^{n+k} \frac{q_{nl}q_{n+k+1,n+k}W^{(q+c)}(l,n+k+1)h(n+k)}{q_{n,n-1}h(n-1)}.
\end{align*}
So it remains to prove that
\begin{equation*}
	W^{(q+c)}(n-1,n+k+1)=\frac{(q_n+q) W^{(q+c)}(n,n+k+1)}{q_{n,n-1}}-\sum_{l=n+1}^{n+k} \frac{q_{nl}W^{(q+c)}(l,n+k+1)}{q_{n,n-1}},
\end{equation*}
by definition, which is equivalent to 
\begin{align}\label{G_n+k+1}
	&\sum_{j=n}^{n+k+1} G_j^{(n+k+1)}(q+c) \nonumber \\
	&=\frac{(q_n+q)}{q_{n,n-1}}\sum_{j=n+1}^{n+k+1}G_j^{(n+k+1)}(q+c)-\frac{1}{q_{n,n-1}}\sum_{l=n+1}^{n+k}\sum_{j=l+1}^{n+k+1} q_{nl}G_j^{(n+k+1)}(q+c).
\end{align}
Indeed, we have
\begin{align*}
	&\sum_{j=n}^{n+k+1} G_j^{(n+k+1)}(q+c)\\
	&=1+\sum_{j=n}^{n+k} \sum_{l=j+1}^{n+k+1} \frac{q_j^{(l)}(q+c)G_l^{(n+k+1)}(q+c)}{q_{j,j-1}}\\
	&=1+\sum_{l=n+1}^{n+k+1}\frac{q_n^{(l)}(q+c)G_l^{(n+k+1)}(q+c)}{q_{n,n-1}}+\sum_{j=n+1}^{n+k} \sum_{l=j+1}^{n+k+1} \frac{q_j^{(l)}(q+c)G_l^{(n+k+1)}(q+c)}{q_{j,j-1}}\\
	&=1+\frac{(q_n+q)\sum_{l=n+1}^{n+k+1}G_l^{(n+k+1)}(q+c)}{q_{n,n-1}}-\sum_{l=n+1}^{n+k+1}\sum_{j=n+1}^{l-1}\frac{q_{nj} G_l^{(n+k+1)}(q+c)}{q_{n,n-1}}\\
	&\qquad -\sum_{l=n+1}^{n+k+1} G_l^{(n+k+1)}(q+c)+\sum_{j=n+1}^{n+k} \sum_{l=j+1}^{n+k+1} \frac{q_j^{(l)}(q+c)G_l^{(n+k+1)}(q+c)}{q_{j,j-1}}\\
	&=\frac{(q_n+q)\sum_{l=n+1}^{n+k+1}G_l^{(n+k+1)}(q+c)}{q_{n,n-1}}-\sum_{j=n+1}^{n+k}\sum_{l=j+1}^{n+k+1}\frac{q_{nj} G_l^{(n+k+1)}(q+c)}{q_{n,n-1}},
\end{align*}
where the last equality follows from $G_{n+k+1}^{(n+k+1)}(q+c)=1$ and
\begin{equation*}
	G_j^{(n+k+1)}(q+c)=\sum_{l=j+1}^{n+k+1} \frac{q_j^{(l)}(q+c)G_l^{(n+k+1)}(q+c)}{q_{j,j-1}}. 
\end{equation*}
So \eqref{G_n+k+1} holds and we complete the proof of \eqref{bar(G)}.

It follows from \eqref{bar(G)} that, 
\begin{equation*}
	\frac{\bar{G}_1^{(N)}}{\sum_{k=1}^{N} \bar{G}_k^{(N)}}=1-\frac{W^{(c)}(1,N)Z^{(c)}(0,\infty)}{W^{(c)}(0,N)Z^{(c)}(1,\infty)},
\end{equation*}
which implies \eqref{killing} by using \eqref{h(i)}. So we complete the proof of this proposition.
\end{proof}

Now we turn to the proof of Theorem \ref{exi_uni}.

\begin{proof}[Proof of Theorem \ref{exi_uni}]
Firstly, we prove the necessity. According to \cite[Proposition 2.4]{CMS13}, if there exists a QSD for $X$ in $E_+$, then $\lambda_0^{(X)}>0$.

Secondly, we prove the sufficiency. According to \cite[Theorem 2.1]{CMS13}, a probability measure $\nu$ (resp. $\bar{\nu}$) on $E_+$ is a QSD for $X$ (resp. $\bar{X}$) in $E_+$ if and only if there exists $\lambda=\lambda(\nu)>0$ (resp. $\lambda=\lambda(\bar{\nu})>0$) such that for any $i,j\geq 1$ and any $t\geq 0$,
\begin{equation*}
	\sum_{i=1}^{\infty} \nu_i p_{ij}(t)=e^{-\lambda t}\nu_j.
\end{equation*}
(resp.
\begin{equation*}
	\sum_{i=1}^{\infty} \bar{\nu}_i \bar{p}_{ij}(t)=e^{-\lambda t}\bar{\nu}_j.)
\end{equation*}
According to \eqref{bar{P}(t)} and for any $i\geq 1$, 
\begin{equation*}
	\frac{1}{Z^{(c)}(0,\infty)}\leq h(i)\leq 1,
\end{equation*}
if $\nu$ is a QSD for $X$, then $\bar{\nu}_i=\frac{\nu_i h(i)}{\sum_{j=1}^{\infty}\nu_j h(j)}$ is a QSD for $\bar{X}$,
and if $\bar{\nu}$ is a QSD for $\bar{X}$, then $\nu_i=\frac{\bar{\nu}_i/h(i)}{\sum_{j=1}^{\infty} \bar{\nu}_j/h(j)}$ is a QSD for $X$.
If $\lambda_0^{(X)}>0$, according to \cite[Proposition 2.4]{CMS13}, 
\begin{equation}\label{lam_com}
	\lambda_0^{(X)}=\liminf_{t\rightarrow \infty}-\frac{1}{t}\log \mathbb{P}_i[T_0>t]\leq \lim_{t\rightarrow+\infty} -\frac{1}{t}\log(p_{ij}(t))=\lambda_0.
\end{equation}
According to \cite[Remark 2.6]{CJL17}, 
\begin{equation*}
	\lambda_0=\sup\{r\geq 0: \exists \text{ positive column } \{f(i), i\in E_+\}, s.t. \  Qf\leq -rf\},
\end{equation*}
where $Qf(i)=\sum_{j=1}^{\infty}q_{ij}f(j)$,
and the exponential decay rate of $\bar{X}$, defined as
\begin{equation*}
	\bar{\lambda}_0=\lim_{t\rightarrow +\infty} -\frac{1}{t} \log(\bar{p}_{ij}(t)),
\end{equation*}
for any $i,j\geq 1$, satisfies 
\begin{equation*}
	\bar{\lambda}_0=\sup\{r\geq 0: \exists \text{ positive column } \{f(i), i\in E_+\}, \text{s.t.} \  \bar{Q}f\leq -rf\}.
\end{equation*}
So $\bar{\lambda}_0=\lambda_0>0$, which implies that the QSD of $\bar{X}$ exists, and then the QSD for $X$ exists.

\end{proof}

\begin{rem}
(1) We claim that $\lambda_0=\lambda_0^{(X)}$. Indeed, from \eqref{lam_com}, if $\lambda_0=0$, then $\lambda_0^{(X)}=0$. If $\lambda_0>0$, according to the proof of Theorem \ref{exi_uni}, there exists a QSD for $X$ satisfying $\mathbb{P}_\nu[T_0>t]=e^{-\lambda_0 t}$, which implies that $\lambda_0^{(X)}\geq \lambda_0$. So $\lambda_0^{(X)}=\lambda_0$.

(2) Assume $\sum_{u=1}^{\infty} c_u W(0,u)<\infty$ and $Y$ satisfies that for any $i\geq 1$, $\mathbb{P}_i[T_0^{(Y)}<\infty]=1$. 
We claim that the exponential decay parameter of $X$ exceeds that of $Y$, that is
\begin{equation*}
	\lambda_0^{(X)}\geq \lambda_0^{(Y)}=\sup\{\lambda\geq 0: \mathbb{E}_i[e^{\lambda T_0^{(Y)}}]<\infty\}=-\lim_{t\rightarrow\infty}\frac{1}{t} \log(p_{ij}^{(Y)}(t)),
\end{equation*}
for any $i,j\geq 1$, where the second equality follows from \cite[Proposition 4.12]{CMS13}. 

Actually, by using \cite[Remark 2.6]{CJL17}, we have
\begin{equation*}
\lambda_0^{(Y)}=\sup\{r\geq 0: \exists \text{ positive column } \{f(i), i\in E_+\}, \text{s.t.} \  Q^{(Y)}f\leq -rf\}.
\end{equation*}
It is easy to show that for any $i\in E_+$  and any positive column $f$ on $E_+$ satisfying $Qf(i)<\infty$, 
\begin{equation*}
Qf(i)=\sum_{j=1}^{\infty}q_{ij}f(j)=\sum_{j=1}^{\infty} q_{ij}^{(Y)} f(j)-c(i)f(i)\leq Q^{(Y)}f(i),
\end{equation*}
which completes the proof of this claim.

So if $Y$ is exponential decay, then the QSD for $X$ exists.
\end{rem}

We are now going to prove Theorem \ref{exp_con}. According to \cite[Theorem 2.1-2.3]{V18}, we obtain the following lemma, which is crucial in the proofs of Theorems \ref{exp_con} and \ref{large_killing}.

\begin{lem}\label{exp_suf}
If there exists a probability measure $\zeta$ on $E_+$ and a constant $k\geq 1$ such that:

(C1) For any $l\geq 1$, there exists $L> l$, $c,t>0$ such that:
\begin{equation*}
	\text{ for any } 1\leq x\leq l, \mathbb{P}_x[X_t\in \d x, t<T_0\wedge T_\partial \wedge T_{L+}]\geq c\zeta(\d x).
\end{equation*}

(C2) Let
\begin{equation*}
	\rho_S=\sup\{\rho\in \mathbb{R}:\sup_{l\geq 1} \inf_{t\geq 0} e^{\rho t} \mathbb{P}_\zeta[t<T_0\wedge T_\partial \wedge T_{l+}]=0 \}.
\end{equation*}
There exists $\rho>\rho_S$ such that:
\begin{equation*}
	\sup_{x>k} \mathbb{E}_x[\exp(\rho (T_k \wedge T_\partial))]<\infty.
\end{equation*}

(C3) It holds that
\begin{equation*}
	\limsup_{t\rightarrow \infty} \sup_{1\leq x\leq k} \frac{\mathbb{P}_x[t<T_0\wedge T_\partial]}{\mathbb{P}_\zeta[t<T_0\wedge T_\partial]} <\infty.
\end{equation*}

Then there exists a unique QSD $\nu$ for $X$ and for any initial distribution $\mu$ on $E_+$, there exist constants $C(\mu)$, $\gamma>0$ such that for any $t>0$,
\begin{equation*}
	||\mathbb{P}_\mu[X_t\in \cdot| t<T_0\wedge T_\partial]-\nu||_{TV}\leq C(\mu)e^{-\gamma t}.
\end{equation*}
\end{lem}

Now we turn to the proof of Theorem \ref{exp_con}.

\begin{proof}[Proof of Theorem \ref{exp_con}]
(1) Firstly, we prove the uniqueness of QSD is equivalent to \eqref{unique}. 

It follows from the proof in Theorem \ref{exi_uni} that
the uniqueness of the QSD for $X$ is equivalent to the uniqueness of the QSD for $\bar{X}$. According to Section 3, the uniqueness of QSD for $\bar{X}$ is equivalent to 
\begin{equation}\label{uni}
	\sum_{i=1}^{\infty} \bar{W}(0,i)=\sum_{i=1}^{\infty} \frac{W^{(c)}(0,i)Z^{(c)}(i,\infty)}{Z^{(c)}(0,\infty)}<\infty,
\end{equation}
where $\bar{W}$ is defined by replacing $q_{ij}$ with $\bar{q}_{ij}$ in \eqref{W^omega} and the first equality follows from \eqref{bar(G)}.
Since for any $j\geq 0$, $1\leq Z^{(c)}(j,\infty)\leq Z^{(c)}(0,\infty)<\infty$, we arrive at
\begin{equation*}
	\sum_{i=1}^{\infty} W^{(c)}(0,i)\frac{Z^{(c)}(i,\infty)}{Z^{(c)}(0,\infty)}\leq \sum_{i=1}^{\infty} W^{(c)}(0,i)\leq \sum_{i=1}^{\infty} W^{(c)}(0,i)Z^{(c)}(i,\infty),
\end{equation*} 
which implies that \eqref{uni} is equivalent to $\sum_{j=1}^{\infty} W^{(c)}(0,j)<\infty$. According to \eqref{W_omega}, we have $W(0,j)\leq W^{(c)}(0,j)$, so 
\begin{equation*}
	\sum_{j=1}^{\infty} W(0,j)\leq \sum_{j=1}^{\infty} W^{(c)}(0,j)<\infty.
\end{equation*}
Since $A=\sum_{u=1}^{\infty} c_u W(0,u)<\infty$, from \eqref{W_c W}, we have
\begin{equation}\label{W W_c}
	\sum_{j=1}^{\infty} W^{(c)}(0,j)\leq \sum_{j=1}^{\infty} W(0,j)e^A\leq e^A \sum_{j=1}^{\infty}W(0,j)<\infty,
\end{equation}
which implies that $\sum_{j=1}^{\infty} W^{(c)}(0,j)<\infty$ is equivalent to \eqref{unique}.

(2) Secondly, we prove \eqref{unique} implies the exponential convergence by using Lemma \ref{exp_suf}. 
According to \eqref{W W_c}, we have
\begin{equation*}
	\sum_{u=1}^{\infty} W^{(c)}(0,u)=\sum_{u=1}^{\infty} \sum_{i=1}^{u} \frac{G_i^{(u)}(c)}{q_{u,u-1}}=\sum_{i=1}^{\infty} \sum_{u=i}^{\infty}\frac{G_i^{(u)}(c)}{q_{u,u-1}}<\infty.
\end{equation*}
So for any $\epsilon>0$, there exists $k_0\geq 1$ such that
\begin{equation*}
	\sum_{i=k_0+1}^{\infty} \sum_{u=i}^{\infty} \frac{G_i^{(u)}(c)}{q_{u,u-1}}=\sum_{u=k_0+1}^{\infty}\sum_{i=k_0+1}^{u} \frac{G_i^{(u)}(c)}{q_{u,u-1}}=\sum_{u=k_0+1}^{\infty} W^{(c)}(k_0,u)<\epsilon.
\end{equation*}
It follows from Theorem \ref{two side exit} that for any $x>k_0$,
\begin{align*}
	\mathbb{E}_x[T_{k_0}\wedge T_\partial]&=\sum_{y=k_0+1}^{\infty} \mathbb{E}_x\[\int_{0}^{T_{k_0}\wedge T_\partial} \mathbf{1}_{[X_t=y]} \d t\]\\
	&=\sum_{y=k_0+1}^{\infty} \(\frac{W^{(c)}(k_0,y)Z^{(c)}(x,\infty)}{Z^{(c)}(k_0,\infty)}-W^{(c)}(x,y)\)\\
	&\leq \sum_{y=k_0+1}^{\infty} W^{(c)}(k_0,y)<\epsilon.
\end{align*}
So 
\begin{equation*}
	\sup_{x>k_0} \mathbb{E}_x[T_{k_0}\wedge T_\partial]<\epsilon.
\end{equation*}
By Chebyshev inequality, 
\begin{equation*}
	\sup_{x>k_0} \mathbb{P}_x[T_{k_0}\wedge T_\partial>1]<\epsilon.
\end{equation*}
It follows from Markov property that for any $n\geq 1$,
\begin{equation*}
	\sup_{x>k_0} \mathbb{P}_x[T_{k_0}\wedge T_\partial>n]<\epsilon^n.
\end{equation*}
By using Fubini theorem, for any $x>k_0$, and $\rho>0$,
\begin{align*}
	\mathbb{E}_x[e^{\rho(T_{k_0}\wedge T_\partial)}]&=\int_{0}^{+\infty} \rho e^{\rho s} \mathbb{P}_x[T_{k_0}\wedge T_\partial>s] \d s+1\\
	&\leq \sum_{n=0}^{\infty} \rho e^{\rho (n+1)} \mathbb{P}_x[T_{k_0}\wedge T_\partial >n]+1\\
	&<\sum_{n=0}^{\infty} \rho e^{\rho(n+1)}\epsilon^n+1. 
\end{align*}
So for any $\rho >0$, we can take $\epsilon<e^{-\rho}$ and find $k_0$ large enough, then
\begin{equation*}
	\sup_{x>k_0} \mathbb{E}_x[e^{\rho (T_{k_0}\wedge T_\partial)}]<\infty.
\end{equation*}

From the irreducibility of $X$ and \cite[Section 3]{V18}, if $\zeta$ is a Dirac measure on $E_+$, then $\rho_S<q_1<\infty$. We take $\rho=q_1+1$ and find $k_0\geq 1$ such that $\sup_{x>k_0} \mathbb{E}_x[e^{(q_1+1) (T_{k_0}\wedge T_\partial)}]<\infty$.

According to Lemma \ref{exp_suf}, we just need to prove (C1)-(C3) by taking $k=k_0$ and $\zeta=\delta_{k_0}$. The above statement has proved (C2) and we are going to prove (C1) and (C3).

According to \cite[Lemma 1.1]{BH00}, since $X$ is irreducible on $E_+$, there exists a sequence $\{L_N\}_{N\geq 1} \subseteq E_+$ such that 
for any $t>0$ and $1\leq i,j\leq L_N-1$,
\begin{equation*}
	p_{ij}^{(L_N)}(t)=\mathbb{P}_i[X_t=j, t<T_0\wedge T_{L_N+}\wedge T_\partial]>0,
\end{equation*}
which implies (C1). 

For any $1\leq x\leq k_0$, by using Markov property,
\begin{align*}
	\mathbb{P}_{k_0}[T_0\wedge T_\partial >t]\geq \mathbb{P}_{k_0}[T_0\wedge T_\partial >t+1]&\geq \mathbb{P}_{k_0}[X_1=x, T_0\wedge T_\partial >t+1]\\
	&=\mathbb{P}_{k_0}[X_1=x] \mathbb{P}_x[T_0\wedge T_\partial >t].
\end{align*}
So we have
\begin{equation}\label{irre}
	\frac{\mathbb{P}_x[T_0\wedge T_\partial >t]}{\mathbb{P}_{k_0}[T_0\wedge T_\partial >t]}\leq \frac{1}{\mathbb{P}_{k_0}[X_1=x]}\leq \frac{1}{\inf_{1\leq x\leq k_0}\mathbb{P}_{k_0}[X_1=x]}<\infty,
\end{equation}
which implies (C3) and we complete the proof of Theorem \ref{exp_con}.
\end{proof}

Based on the proofs of Theorems \ref{exi_uni} and \ref{exp_con}, as well as Proposition \ref{no-killing unique} and Lemma \ref{QSD_Y}, we obtain the following corollary, which provides an explicit representation of the QSD for $X$.

\begin{cor}
Suppose the assumptions in Theorem \ref{exi_uni} holds. 

(1) If $\lambda_0^{(X)}=0$, then there exists no QSD for $X$.

(2) If $\lambda_0^{(X)}>0$ and $\sum_{u=1}^{\infty} W(0,u)=\infty$, then there exists a continuum of QSDs $\{\nu^\theta\}_{\theta\in (0,\lambda_0^{(X)}]}$ for $X$, where
\begin{equation*}
	\nu^{\theta}_i=\frac{W^{(c-\theta)}(0,i)}{\sum_{j=1}^{\infty}W^{(c-\theta)}(0,j)}, \qquad i\geq 1.
\end{equation*}

(3) If $\sum_{u=1}^{\infty}W(0,u)$, then $\lambda_0^{(X)}>0$ and there exists a unique QSD $\nu$ for $X$, 
where 
\begin{equation*}
	\nu_i=\frac{W^{(c-\lambda_0^{(X)})}(0,i)}{\sum_{j=1}^{\infty} W^{(c-\lambda_0^{(X)})}(0,j)}, \qquad i\geq 1.
\end{equation*}
\end{cor}

In recent years, there has been an increasing amount of literature on the study of one-dimensional Markov processes at their boundaries. 
This study provides key insights into the behavior of processes when their initial values are sufficiently large and has extensive applications in population modeling. The boundary is called an entrance boundary if the dynamics inhibit the process from hitting $\infty$ in finite time and permit the process to ``start from infinity". Consequently, the process is said to comes down from infinity. 

Coming down from infinity is also an important property in the study of QSDs for birth and death processes and diffusion processes, of which the characterizations have been given in \cite{A91} and \cite{CCLMMS09} respectively. Recall that in Proposition \ref{no-killing unique}, we call a single death process without killing $Y$ comes down from infinity if there exists $t_0>0$ such that
\begin{equation*}
	\lim_{x\rightarrow\infty}\mathbb{P}_x[T_0^{(Y)}<t_0]>0,
\end{equation*} 
which is consistent with ``coming down from infinity instantaneously" defined in \cite{FLZ20} and equivalent to $\sum_{u=1}^{\infty}W(0,u)<\infty$. According to Proposition \ref{no-killing unique}, this property determines the uniqueness of QSD.

In \cite{CV17}, a comparable concept referred to as ``coming down from infinity before killing" is introduced and employed to provide a sufficient condition for the uniform exponential convergence in the total variation norm to a QSD for diffusion processes with killing. Similar to \cite{CV17}, we call a single death process with killing $X$ comes down from infinity before killing if there exists $t_0>0$ such that
\begin{equation*}
	\lim_{x\rightarrow\infty} \mathbb{P}_x[T_0<t_0\wedge T_\partial]>0.
\end{equation*}

The following proposition provides a sufficient and necessary condition for $X$ coming down from infinity before killing.

\begin{prop}
Suppose the non-killing process $Y$ corresponding to $X$ satisfies $\mathbb{P}_i[T_0^{(Y)}<\infty]=1$ for any $i\geq 1$. Then $X$ comes down from infinity before killing if and only if $\sum_{u=1}^{\infty}(1+c_u)W(0,u)<\infty$.
\end{prop}

\begin{proof}
(1) Firstly, we show that if $\sum_{u=1}^{\infty} (1+c_u)W(0,u)<\infty$, then $X$ comes down from infinity before killing.

According to \eqref{h(i)}, since $\sum_{u=1}^{\infty}c(u)W(0,u)<\infty$, we get that
\begin{equation*}
	\lim_{x\rightarrow\infty} \mathbb{P}_x[T_0<T_\partial]\geq \frac{1}{Z^{(c)}(0,\infty)}>0.
\end{equation*}
From the construction of $X$ by $Y$ and $c$, we have
\begin{equation*}
	\mathbb{P}_x[T_0<t\wedge T_\partial]
	=\mathbb{P}_x[T_0^{(Y)}<t,T_0^{(Y)}<T_\partial].
\end{equation*}
By taking $t_0=t_{1/Z^{(c)}(0,\infty)}$ in the first part of the proof of Proposition \ref{no-killing unique}, we have
\begin{equation*}
	\lim_{x\rightarrow\infty} \mathbb{P}_x[T_0^{(Y)}>t_0]<\frac{1}{Z^{(c)}(0,\infty)},
\end{equation*}
which implies that
\begin{equation*}
	\lim_{x\rightarrow\infty} \mathbb{P}_x[T_0<t_0\wedge T_\partial]\geq \lim_{x\rightarrow\infty} \mathbb{P}_x[T_0<T_\partial]-\lim_{x\rightarrow\infty}\mathbb{P}_x[T_0^{(Y)}>t_0]>0.
\end{equation*}
So we complete the proof of the sufficiency.

(2) Secondly, we show that if $X$ comes down from infinity before killing, then $\sum_{u=1}^{\infty}(1+c_u)W(0,u)<\infty$.

According to the construction of $X$ by $Y$ and $c$, we have
\begin{equation*}
	\mathbb{P}_x[T_0<t_0\wedge T_\partial]=\mathbb{E}_x\[\exp\(- \int_{0}^{T_0^{(Y)}} c(Y_s) \d s\)\mathbf{1}_{[T_0^{(Y)}<t_0]} \].
\end{equation*}
So 
\begin{align*}
	&\mathbb{E}_x \[\exp\(-\int_{0}^{T_0^{(Y)}} (c(Y_s)+1) \d s  \)\]\\
	&\geq \mathbb{E}_x \[ \exp\( -\int_{0}^{T_0^{(Y)}} (c(Y_s)+1) \d s \) \mathbf{1}_{[T_0^{(Y)}<t_0]} \]\\
	&\geq e^{-t_0} \mathbb{E}_x\[\exp\(- \int_{0}^{T_0^{(Y)}} c(Y_s) \d s\)\mathbf{1}_{[T_0^{(Y)}<t_0]} \],
\end{align*}
which implies that
\begin{equation*}
	\lim_{x\rightarrow\infty} \mathbb{E}_x \[\exp\(-\int_{0}^{T_0^{(Y)}} (c(Y_s)+1) \d s  \)\]>0.
\end{equation*}

We define a time-changed process of $Y$. Let 
\begin{equation*}
	A_t=\int_{0}^{t}c(Y_s)\d s+t, \qquad \gamma_t=\inf\{s\geq 0: A_s>t\}, \qquad t\geq 0.
\end{equation*}
Define
\begin{equation*}
	\tilde{Y}_t=Y_{\gamma_t}, \qquad t\geq 0.
\end{equation*}
Then $\tilde{Y}$ is a conservative single death process on $E$ with transition rate matrix $\tilde{Q}=\(\tilde{q}_{ij}\)_{i,j\geq 0}$, where $\tilde{q}_{ij}=\frac{q_{ij}^{(Y)}}{c(i)+1}$. According to \eqref{W^omega}, we define $\tilde{W}(i,j)$ by replacing $q_{ij}$ with $\tilde{q}_{ij}$. Then $\tilde{W}(0,u)=(1+c(u))W(0,u)$. 

Given that the embedding chain for $\tilde{Y}$ is identical to that for $Y$, we have $\mathbb{P}_i[\tilde{T}_0<\infty]=1$ for any $i\geq 1$, where $\tilde{T}_0=\inf\{t\geq 0: \tilde{Y}_t=0\}$. Then
\begin{equation*}
	\mathbb{E}_x[e^{-\tilde{T}_0}]=\mathbb{E}_x\[\exp\(-\int_{0}^{T_0} (c(Y_s)+1) \d s\)\].
\end{equation*}
So $\lim_{x\rightarrow\infty} \mathbb{E}_x[e^{-\tilde{T}_0}]>0$. By Fubini theorem,
\begin{equation*}
	\mathbb{E}_x[e^{-\tilde{T}_0}]=1-\int_{0}^{\infty} e^{-s} \mathbb{P}_x[\tilde{T}_0>s] \d s=\int_{0}^{\infty} e^{-s}\mathbb{P}_x[\tilde{T}_0<s] \d s.
\end{equation*}
By using the monotonicity of $\mathbb{P}_\cdot [\tilde{T}_0>t]$, we have
\begin{equation*}
	\lim_{x\rightarrow\infty}\int_{0}^{\infty} e^{-s}\mathbb{P}_x[\tilde{T}_0<s] \d s=\int_{0}^{\infty} e^{-s} \lim_{x\rightarrow\infty} \mathbb{P}_x[\tilde{T}_0<s]\d s>0.
\end{equation*}
So there exists $s_0>0$ such that $\lim_{x\rightarrow\infty} \mathbb{P}
_x[\tilde{T}_0<s_0]>0$. From the proof of Proposition \ref{no-killing unique}, we have
\begin{equation*}
	\sum_{u=1}^{\infty} \tilde{W}(0,u)=\sum_{u=1}^{\infty} (1+c(u))W(0,u)<\infty.
\end{equation*}
This completes the proof of this proposition.
\end{proof}

Based on an argument similar to the former proof, we obtain the probabilistic interpretation of the ``small killing" condition: $\sum_{u=1}^{\infty} c(u) W(0,u)<\infty$.

\begin{prop}\label{small killing}
Suppose the non-killing process $Y$ satisfies $\mathbb{P}_i[T_0^{(Y)}<\infty]=1$ for any $i\geq 1$. Then
\begin{equation*}
	\lim_{x\rightarrow\infty} \mathbb{P}_x[T_0<T_\partial]>0,
\end{equation*}
if and only if the ``small killing" condition holds, that is,
\begin{equation*}
	\sum_{u=1}^{\infty} c(u) W(0,u)<\infty.
\end{equation*}
\end{prop}

\begin{proof}
On the one hand, it follows from \eqref{h(i)} that if $\sum_{u=1}^{\infty}c(u)W(0,u)<\infty$, then
\begin{equation*}
	\lim_{x\rightarrow\infty} \mathbb{P}_x[T_0<T_\partial]\geq \frac{1}{Z^{(c)}(0,\infty)}>0.
\end{equation*}

On the other hand, we assume 
\begin{equation*}
	E_c=\{x\geq 1: c(x)>0\}
\end{equation*}
is an infinite set without loss of generality and we denote $E_c=\{x_n\}_{n\geq 1}$ with $x_i<x_{i+1}$ for any $i\geq 1$. And we define a time-changed process of $Y$ on $E_c$ as follows. Define an additive functional 
\begin{equation*}
	\bar{A}_t=\int_{0}^{t} c(Y_s) \d s,\quad \text{ for } t<T_0^{(Y)},
\end{equation*}
and let $\bar{A}_\infty=\lim_{t\rightarrow \infty}\bar{A}_t$ on the event $\{T_0^{(Y)}=\infty\}$.
Its right inverse function is defined as
\begin{equation*}
	\bar{\gamma}_t=\inf\{s>0:\bar{A}_s>t\},\quad \text{ for } t<\bar{A}_{T_0^{(Y)}}.
\end{equation*}
Then the process $\bar{Y}$ is defined, stopped at time $\bar{T}_0^{(Y)}\leq \infty$, by letting 
\begin{equation*}
	\bar{Y}_t=Y_{\bar{\gamma}_t},\quad \text{ for } t<\bar{A}_{T_0^{(Y)}}.
\end{equation*} 
Define the first passage time of $\bar{Y}$ by
\begin{equation*}
	\bar{T}_a=\inf\{t\geq 0: \bar{Y}_t=a\}, \quad  \bar{T}_{a+}=\inf\{t\geq 0: \bar{Y}_t\geq a\},
\end{equation*}
for $a\in E_c$ and $\bar{T}_\infty=\lim_{n\rightarrow \infty}\bar{T}_{x_n+}$
with the convention $\inf\emptyset=\infty$. It follows from the results for time-changed process that the first passage time
\begin{equation*}
	\bar{T}_a=\bar{A}_{T_a^{(Y)}} \text{ on the event }\{T_a^{(Y)}<\infty\},\quad \bar{T}_{a+}=\bar{A}_{T_{a+}^{(Y)}} \text{ on the event } \{T_{a+}^{(Y)}<T_0^{(Y)}\}.
\end{equation*}
So the lifetime of $\bar{Y}$ can be regarded as
\begin{equation*}
	\bar{T}_0=
	\begin{cases}
		\bar{A}_{T_0^{(Y)}}& \text{ on the event } \{T_0^{(Y)}<\infty\}\\
		\bar{T}_{\infty}& \text{ on the event }\{T_0^{(Y)}=\infty\}.
	\end{cases}
\end{equation*}
From the skip-free property of $Y$ and $\mathbb{P}_i[T_0^{(Y)}<\infty]=1$ for any $i\geq 1$, we have $\bar{Y}$ is a single death process on $E_c$.
According to the construction of $X$ by $Y$ and $c$, the proof of the above remark and Proposition \ref{no-killing unique}, we have 
\begin{equation*}
	\lim_{n\rightarrow\infty}\mathbb{P}_{x_n}[T_0<T_\partial]=\lim_{n\rightarrow\infty} \mathbb{E}_{x_n}[\exp(-\bar{A}_{T_0^{(Y)}})]>0 
\end{equation*}
is equivalent to 
\begin{equation*}
	\lim_{n\rightarrow\infty}\mathbb{E}_{x_n}[\bar{A}_{T_0^{(Y)}}]<\infty.
\end{equation*}
According to \eqref{measure2}, we get that
\begin{align*}
	\mathbb{E}_x[\bar{A}_{T_0^{(Y)}}]=\sum_{u\in E_c} \mathbb{E}_x\[\int_{0}^{\bar{A}_{T_0^{(Y)}}} \mathbf{1}_{[\bar{Y}_s=u]} \d s\]&=\sum_{u\in E_c}c(u) \mathbb{E}_x\[\int_{0}^{T_0^{(Y)}} \mathbf{1}_{[Y_s=u]} \d s\]\\
	&=\sum_{u\in E_c} c(u)(W(0,u)-W(x,u)).
\end{align*}
By monotone convergence theorem, we have
\begin{equation*}
	\lim_{n\rightarrow\infty}\mathbb{E}_{x_n}[\bar{A}_{T_0^{(Y)}}]=\sum_{u\in E_c}c(u)W(0,u),
\end{equation*}
which completes the proof.
\end{proof}

\subsection{``Bounded killing" case: proof of Theorem \ref{uni_exp_con}}

By noting that Theorems \ref{exi_uni} and \ref{exp_con} do not yield a sufficient condition for the uniform exponential convergence of the conditional distribution to QSD of $X$ in $E_+$ in the total variation norm, whereas Theorem \ref{uni_exp_con} provides a sufficient and necessary condition under the assumption of bounded killing, that is, $\sup_{n\geq 1} c_n<\infty$. In this subsection, we also assume that the non-killing process $Y$ corresponding to $X$ satisfies $\mathbb{P}_i[T_0^{(Y)}<\infty]=1$ for any $i\geq 1$.

We are now turning to the proof of Theorem \ref{uni_exp_con}. The idea of the proof benefits from \cite[Theorem 2.1]{CV16}.

\begin{proof}[Proof of Theorem \ref{uni_exp_con}]
According to \cite[Theorem 2.1]{CV16}, the uniform exponential convergence to QSD is equivalent to the condition that there exists a probability measure $\nu$ such that 

(A1) there exists $t_0,C_1>0$ such that for any $x\geq 1$,
\begin{equation*}
	\mathbb{P}_x[X_{t_0}\in \cdot| t_0< T_0\wedge T_\partial]\geq C_1\nu(\cdot);
\end{equation*}

(A2) there exists $C_2>0$ such that for any $x\geq 1$ and any $t\geq 0$,
\begin{equation*}
	\mathbb{P}_\nu[T_0\wedge T_\partial>t]\geq C_2\mathbb{P}_x[T_0\wedge T_\partial >t].
\end{equation*}

Firstly, we prove the sufficiency. 

Assume $\sum_{u=1}^{\infty} W(0,u)<\infty$. According to \cite[Proposition 2.5]{Z18}, for the non-killing process $Y$ corresponding to $X$, we obtain that for any $\epsilon>0$, there exists $k_1$ large enough such that 
\begin{equation*}
	\sup_{x>k_1} \mathbb{E}_x[T_{k_1}^{(Y)}]<\epsilon.
\end{equation*}
Then it follows from an argument similar to the proof of (C2) in Theorem \ref{exp_con} that for $\rho=\sup_n c_n+q_{10}+1$, we can find $k_1$ large enough and let $K=\{1,2,\cdots, k_1\}$, $\tau_K=\inf\{t\geq 0:X_t\in K\}$ such that 
\begin{equation}\label{sup E_x}
	A=\sup_{x\geq 1}\mathbb{E}_x[e^{\rho(\tau_K\wedge T_\partial)}]\leq \sup_{x> k_1} \mathbb{E}_x[e^{\rho T_{k_1}^{(Y)}}]<\infty.
\end{equation}

Now, we are going to prove (A1) and (A2) with $\nu=\delta_{k_1}$. 
Since $X$ is irreducible, by using \eqref{irre} and let $C^{-1}=\inf_{x\in K} \mathbb{P}_{k_1}[X_1=x]>0$, we have that for any $t\geq 0$,
\begin{equation*}
	\sup_{x\in K} \mathbb{P}_x[T_0\wedge T_\partial>t]\leq C\mathbb{P}_{k_1}[T_0\wedge T_\partial>t].
\end{equation*}
Since $\rho > \sup_n c_n+q_{10}$, we have that for any $s\geq 0$, $\inf_{x\geq 1} \mathbb{P}_x[T_0\wedge T_\partial>s]\geq e^{-\rho s}$.
Then by using Markov property, for any $t\geq s\geq 0$,
\begin{equation*}
	\mathbb{P}_{k_1}[T_0\wedge T_\partial >t]\geq e^{-\rho s}\mathbb{P}_{k_1}[T_0\wedge T_\partial >t-s].
\end{equation*}
According to \eqref{sup E_x}, by using Chebyshev inequality, we obtain
\begin{equation*}
	\mathbb{P}_x[t<\tau_K \wedge T_\partial]\leq Ae^{-\rho t}.
\end{equation*}
By using strong Markov property, we have
\begin{align*}
	\mathbb{P}_x[T_0\wedge T_\partial >t]&=\mathbb{P}_x[t<\tau_K\wedge T_0\wedge T_\partial]+\mathbb{P}_x[\tau_K\leq t<T_0\wedge T_\partial]\\
	&\leq Ae^{-\rho t}+\int_{0}^{t} \sup_{y\in K} \mathbb{P}_y[t-s<T_0\wedge T_\partial]\mathbb{P}_x[\tau_K\wedge T_\partial \in \d s]\\
	&\leq A\mathbb{P}_{k_1}[T_0\wedge T_\partial>t]+C\int_{0}^{t} \mathbb{P}_{k_1} [t-s<T_0\wedge T_\partial]\mathbb{P}_x[\tau_K\wedge T_\partial \in \d s]\\
	&\leq A\mathbb{P}_{k_1}[T_0\wedge T_\partial>t]+C\mathbb{P}_{k_1}[t<T_0\wedge T_\partial]\int_{0}^{t} e^{\rho s}\mathbb{P}_x[\tau_K\wedge T_\partial \in \d s]\\
	&\leq A(1+C) \mathbb{P}_{k_1}[t<T_0\wedge T_\partial],
\end{align*}
which implies (A2).

Then we prove (A1). For any $x\geq 1$,
\begin{equation*}
	\mathbb{P}_x[\tau_K<t]\geq \mathbb{P}_x[\tau_K<t<T_\partial]=\mathbb{P}_x[t<T_\partial]-\mathbb{P}_x[t<\tau_K\wedge T_\partial]\geq e^{-(\rho-1)t}-Ae^{-\rho t}.
\end{equation*}
So there exists $t_0>0$ such that 
\begin{equation*}
	\inf_{x\geq 1}\mathbb{P}_x[\tau_K\leq t_0-1]>0.
\end{equation*}
By using strong Markov property and irreducibility of $X$, we have
\begin{align*}
	C_1=\inf_{x\geq 1} \mathbb{E}_x[\mathbf{1}_{[\tau_K<t_0-1]} \inf_{y\in K}\mathbb{P}_y[X_1=k_1]e^{-q_{k_1}(t_0-1-\tau_K)}]&\leq \inf_{x\geq 1} \mathbb{P}_x[X_{t_0}=k_1]\\
	&\leq \inf_{x\geq 1} \mathbb{P}_x[X_{t_0}=k_1|t_0<T_0\wedge T_\partial],
\end{align*}
which implies (A1). And we complete the proof of the sufficiency.

Secondly, we prove the necessity. Considering that $\mathbb{P}_i[T_0^{(Y)}<\infty]=1$ for any $i\geq 1$,  
it follows from \cite[Proposition 2.5]{Z18} or \cite[Proposition A.1]{Y22} that $S=\sup_{n\geq 1}\mathbb{E}_n[T_0^{(Y)}]=\sum_{u=1}^{\infty} W(0,u)$. 

Since the uniform exponential convergence holds, by using \cite[Theorem 2.1]{CV16}, there exists a probability measure $\nu$ on $E_+$ and $t_0, C_1>0$ such that for any $x\geq 1$, 
\begin{equation*}
	\mathbb{P}_x[X_{t_0}\in \cdot|t_0<T_0\wedge T_\partial]\geq C_1\nu(\cdot).
\end{equation*}
So there exists $k_2>0$ such that, $C_2=\nu({\{1,2,\cdots,k_2\}})>0$. Since $\sup_n c_n<\infty$, we have
\begin{equation*}
	\inf_{x\geq 1}\mathbb{P}_x[Y_{t_0}\leq k_2]\geq C_1\nu({\{1,2,\cdots,k_2\}})\inf_{x\geq 1}\mathbb{P}_x[T_0\wedge T_\partial>t_0]\geq C_1 C_2 e^{-(\rho-1)t_0}.
\end{equation*}
So for any $x>k_2$,
\begin{equation*}
	\mathbb{P}_x[T_{k_2}^{(Y)}\geq t_0]\leq \mathbb{P}_x[Y_{t_0}>k_2]\leq 1-C_1 C_2e^{-(\rho-1)t_0}.
\end{equation*}
By using Markov property, for any $n\geq 1$,
\begin{align*}
	\mathbb{P}_x[T_{k_2}^{(Y)}>(n+1)t_0]\leq \mathbb{P}_x[T_{k_2}^{(Y)}>n t_0] \sup_{y> k_2}\mathbb{P}_y[T_{k_2}^{(Y)}\geq t_0]&\leq  \mathbb{P}_x[T_{k_2}^{(Y)}>n t_0] (1-C_1 C_2e^{-(\rho-1)t_0})\\
	&\leq (1-C_1 C_2e^{-(\rho-1)t_0})^{n+1}.
\end{align*}
By using Fubini theorem, for any $x\geq k_2$, 
\begin{equation*}
	\mathbb{E}_x[T_{k_2}^{(Y)}]=\int_{0}^{\infty} \mathbb{P}_x[T_{k_2}^{(Y)}>t] \d t\leq \sum_{n=0}^{\infty} t_0 \mathbb{P}_x[T_{k_2}^{(Y)}>n t_0] \leq t_0\sum_{n=0}^{\infty} (1-C_1 C_2e^{-(\rho-1)t_0})^n<\infty.
\end{equation*}
So $M_{k_2}=\sup_{x\geq k_2+1} \mathbb{E}_x[T_{k_2}^{(Y)}]<\infty$. By using strong Markov property and skip-free property, we have
\begin{equation*}
	\mathbb{E}_{k_2}[T_{k_2-1}^{(Y)}]=\frac{1}{q_{k_2}^{(Y)}}+\sum_{k\geq k_2+1} \frac{q_{k_2,k}}{q_{k_2}^{(Y)}} \mathbb{E}_{k}[T_{k_2-1}^{(Y)}]\leq \frac{1}{q_{k_2}^{(Y)}}+\sum_{k\geq k_2+1} \frac{q_{k_2,k}}{q_{k_2}^{(Y)}}(M_{k_2}+\mathbb{E}_{k_2}[T_{k_2-1}^{(Y)}]),
\end{equation*}
which implies that $\mathbb{E}_{k_2}[T_{k_2-1}^{(Y)}]<\infty$. So $M_{k_2-1}=\sup_{x\geq k_2}\mathbb{E}_x[T_{k_2-1}^{(Y)}]=M_{k_2}+\mathbb{E}_{k_2}[T_{k_2-1}^{(Y)}]<\infty$. By induction, we obtain that
$S=\sup_{x\geq 1}\mathbb{E}_x[T_0^{(Y)}]=\sum_{u=1}^{\infty} W(0,u)<\infty$, and we complete the proof of the necessity.
\end{proof}

\subsection{``Large killing" case: proof of Theorem \ref{large_killing}}

Recall that the large killing condition is: $\liminf_{n\rightarrow \infty}c_n>\inf_{n\geq 1} q_n$ in this paper. The large killing condition is satisfied in some concrete examples, such as the linear killing case in \cite{ZZ13} and \cite{CNZ14}.

We are now going to prove Theorem \ref{large_killing}. The idea of the proof benefits from \cite[Theorem 4.1]{V18}.

\begin{proof}[Proof of Theorem \ref{large_killing}]
Since 
\begin{equation*}
	\liminf_{n\rightarrow \infty} c_n>\inf_{n\geq 1} q_n,
\end{equation*}
there exists $k_1,l_1>0$ such that 
\begin{equation*}
	0<q_{k_1}<\inf_{n\geq l_1+1} c_n.
\end{equation*}
We will prove the conditions in Lemma \ref{exp_suf} with $\zeta=\delta_{k_1}$ and $k=l_1$. (C1) and (C3) can be proved similar to the proof in Theorem \ref{exp_con}, so we omit it. It remains to prove (C2) for $q_{k_1}< \rho<\inf_{n\geq l_1+1} c_n$. From the irreducibility of $X$ and \cite[Section 3]{V18}, $\rho_S<q_{k_1}<\rho$. For any $x\geq l_1+1$ and any $t>0$,
\begin{equation*}
	\mathbb{P}_x[T_{l_1}\wedge T_\partial >t]\leq e^{-\tilde{\rho} t},
\end{equation*}
where $\tilde{\rho}=\inf_{n\geq l_1+1} c_n$, which implies that
\begin{equation*}
	\mathbb{E}_x[e^{\rho(T_{l_1}\wedge T_\partial)}]=1+\int_{0}^{\infty} \rho e^{\rho s} \mathbb{P}_x[T_{l_1}\wedge T_\partial >s] \d s\leq 1+\frac{\rho}{\tilde{\rho}-\rho}.
\end{equation*}
So (C2) holds and we complete the proof of Theorem \ref{large_killing}.
\end{proof}

%----------------------------------------------------------------------------------------------
%----------------------------------------------------------------------------------------------

{\bf Acknowledgement}\
This work was supported by the National Nature Science Foundation of China (Grant No. 12171038), National Key Research and Development Program of China (2020YFA0712900).

{\bf Author Contributions}\
All authors wrote the main manuscript text. All authors reviewed the manuscript.

{\bf Conflict of interest}\
The authors have no competing interests to declare that are relevant to the content of this article.

%%%%%%%%%%%%%%%%%%%%%%%%%%%%%%%%%%%%%%%%%%%%%%%%%%%%%%%%%%%%%%

\end{document}